\documentclass[12pt]{amsart}

\usepackage{amsmath}
\usepackage{amssymb}
\usepackage{amscd}
\usepackage{bm}
\usepackage{graphicx}
\numberwithin{equation}{section}
\allowdisplaybreaks[1]
\newcommand{\comment}[1]{}

\usepackage{amsthm}
\theoremstyle{plain}
 \newtheorem{theorem}{Theorem}[section]
 \newtheorem{proposition}[theorem]{Proposition}
 \newtheorem{lemma}[theorem]{Lemma}
 \newtheorem{corollary}[theorem]{Corollary}
 \newtheorem{fact}[theorem]{Fact}
\theoremstyle{remark}
 \newtheorem{remark}[theorem]{Remark}

\usepackage[usenames]{color}
\def\dsp{\displaystyle}
\def\rem#1{}

\newcommand{\C}{\mathbb C}

\newcommand{\R}{\mathbb R}

\renewcommand{\P}{\mathbb P}

\newcommand{\Z}{\mathbb Z}

\newcommand{\CE}{\mathcal E}

\newcommand{\CM}{\mathcal M}
\newcommand{\cCM}{\widetilde{\mathcal M}}
\newcommand{\wM}{\widetilde{M}}

\newcommand{\im}{\mathrm{Im}}

\newcommand{\g}{\gamma}
\newcommand{\De}{\Delta}
\newcommand{\cD}{\widetilde{\Delta}}

\newcommand{\e}{\varepsilon}

\newcommand{\xx}{\mathbf x}
\newcommand{\yy}{\mathbf y}
\newcommand{\zz}{\mathbf z}
\newcommand{\Hl}{H^{\mathrm{lf}}}
\newcommand{\fI}{\mathfrak{I}}

\begin{document}

\title[Schwarz map of reducible $E_2$]
{An example of Schwarz map of reducible hypergeometric equation $E_2$ 
in two variables }

\author{Keiji \textsc{Matsumoto}}
\address{Department of Mathematics, Hokkaido University, 
Sapporo  060-0810 Japan}
\email{matsu@math.sci.hokudai.ac.jp}
\author{Takeshi \textsc{Sasaki}}
\address{Kobe University,   Kobe 657-8501, Japan}
\email{sasaki@math.kobe-u.ac.jp}
\author{Tomohide \textsc{Terasoma}}
\address{Department of Mathematics Tokyo University, 
Tokyo 153-8914 Japan}
\email{terasoma@ms.u-tokyo.ac.jp} 
\author{Masaaki \textsc{Yoshida}}
 \address{Kyushu University,   Fukuoka 819-0395, Japan}
\email{myoshida@math.kyushu-u.ac.jp}


\begin{abstract} We study an Appell hypergeometric system $E_2$ of rank four 
which is reducible and show that its Schwarz map admits 
geometric interpretations: the map can be considered as 
the universal Abel-Jacobi map of a $1$-dimensional family of curves of genus 2.
\end{abstract}

\maketitle

\tableofcontents
\par\bigskip

\section*{Introduction}

Schwarz maps for hypergeometric systems in single and several variables 
are studied by several authors (cf. \cite{Yo}) for more than hundred years. 
These systems treated were irreducible, 
maybe because specialists believed that reducible systems would not give 
interesting Schwarz maps. 

We study in this paper Appell's hypergeometric system $E_2$ of 
rank four when its parameters satisfy $a-c\in \Z$ or $a-c'\in \Z$.  
In this case, the system $E_2$ is reducible, and 
has a $3$-dimensional subsystem isomorphic to Appell's $E_1$ 
(Proposition \ref{prop:s2}).
If $a-c,a-c'\in \Z$ then $E_2$ has two such subsystems. 
By Proposition \ref{prop:s2g},  
the intersection of these subsystems 
is equal to the Gauss hypergeometric equation.
As a consequence, we have inclusions on $E_2$, two $E_1$'s
and $E$ (Theorem \ref{matome}).

We give the monodromy representation of the system $E_2$
which can be specialized to the case $a-c,a-c'\in \Z$ 
in Theorem \ref{th:monod-rep}.
As for explicit circuit matrices 
with respect to a basis $\De_1,\dots,\De_4$, see 
Corollary \ref{cor:monod-matrix}.

We further specialize the parameters of the system $E_2$ as 
$$(a,b,b',c,c')
=\big(\dfrac{4}{3},\dfrac{2}{3},\dfrac{2}{3},\dfrac{4}{3},\dfrac{4}{3}\big)$$
in \S \ref {Schmap}.  
In this case, the restriction of its monodromy group 
to the invariant subspace is arithmetic and isomorphic to 
the triangle group of type $[3,\infty,\infty]$. 
We show that its Schwarz map admits geometric interpretations: 
the map can be considered as the universal Abel-Jacobi map 
of a 1-dimensional family of curves of genus 2 in 
Theorem \ref{th:gen-Schwarz}.

The system $E_2$ is equivalent to the restriction of a 
hypergeometric system $E(3,6;a_1,\dots,a_6)$ to 
a two dimensional stratum in the configuration 
space $X(3,6)$ of six lines in the projective plane.
In Appendix \ref{3-dim-S}, we study  a system of hypergeometric differential 
equations in three variables, which is obtained by restricting 
$E(3,6;a_1,\dots,a_6)$ to the three dimensional strata  corresponding to 
configurations only with one triple point. 
The methods to prove Proposition \ref{prop:s2} 
are also applicable to this system under a reducibility condition. 
In Appendix \ref{genus2}, we classify families of genus $2$ branched coverings 
of the projective line, whose period maps yield triangle groups.

In a forthcoming paper \cite{MT}, we study this Schwarz map 
using period domains for Mixed Hodge structures. 
Moreover, we explicitly give its inverse in terms of theta functions.  
 
\section{Some generalities on Appell's systems $E_2$ and $E_1$}
\label{sec:Appell-E1-E2}
Gauss hypergeometric series
$$F(a,b,c;x)=\sum_{i=0}\dsp\frac{(a,i)(b,i)}{(c,i)\ i!}x^i,$$
where $(a,i)=a(a+1)\cdots(a+n-1)$, admits an integral representation:
$$\begin{array}{ll}
\dsp\frac{\Gamma(a)\Gamma(c-a)}{\Gamma(c)}F(a,b,c ;x)&=\dsp\int_0^1t^{a-1}(1-t)^{c-a-1}(1-tx)^{-b}dt\\&={\rm constant}\times\dsp\int_\infty^1t^{b-c}(1-t)^{c-a-1}(x-t)^{-b}dt.\end{array} $$
The function $z=F$ is a solution of the hypergeometric equation 
\[E(a,b,c;x):\qquad  P(a,b,c;x)z=0,\]
where
\[P(a,b,c;x)=D(c-1+D)-x(a+D)(b+D),\qquad D=x\dsp{d\over d x}.\]
The collection of solutions is denoted by $S(a,b,c;x)$.

\par\medskip
Appell's hypergeometric series
$$F_1(a,b,b',c;x,y)=\sum_{i,j=0}\dsp\frac{(a,i+j)(b,i)(b',j)}{(c,i+j)\ i!\ j!}x^iy^j,$$
 admits an integral representation:
$$\frac{\Gamma(a)\Gamma(c-a)}{\Gamma(c)}F_1(a,b,b',c ;x,y)=\int_0^1t^{a-1}(1-t)^{c-a-1}(1-tx)^{-b}(1-ty)^{-b'}dt. $$
The function $z=F_1$ is a solution of the hypergeometric system 
\[E_1(a,b,b',c;x,y):\quad \left\{
\begin{array}{l}
 (D(c-1+D+D')-x(a+D+D')(b+D))z=0, \\[2mm]
 (D'(c-1+D+D')-y(a+D+D')(b'+D'))z=0, 
\end{array}
\right.
\]
where $D=x\partial/\partial x, D'=y\partial/\partial y$, 
which can be written as 
$$P_1(a,b,b',c;x,y)z=Q_1(a,b,b',c;x,y)z=R_1(a,b,b',c;x,y)z=0,$$
where
\begin{eqnarray*}
&P_1(a,b,b',c;x,y)=x(1-x)\partial_{xx}+y(1-x)\partial_{xy}
    +(c-(a+b+1)x)\partial_x-b y\partial_y-ab,& \\[2mm]
&Q_1(a,b,b',c;x,y)=y(1-y)\partial_{yy}+x(1-y)\partial_{yx}
    +(c-(a+b'+1)y)\partial_y-b'x\partial_x-ab', &\\[2mm]
&R_1(a,b,b',c;x,y)=(x-y)\partial_{xy}-b'\partial_x+b\partial_y, 
\end{eqnarray*}
and $\partial_x=\partial/\partial x$, etc. 
The last equation $R_1(a,b,b',c;x,y)z=0$
is derived from the integrability condition of the first two equations.
The collection of solutions is denoted by $S_1(a,b,b',c;x,y)$.
\par\medskip
Appell's hypergeometric series 
$$F_2(a,b,b',c,c';x,y)=\sum_{i,j=0}\dsp\frac{(a,i+j)(b,i)(b',j)}{(c,i)(c',j)\ i!\ j!}x^iy^j$$
admits an integral representation:
$$\frac{\Gamma(b)\Gamma(b')\Gamma(c-b)\Gamma(c'-b')}{\Gamma(c)\Gamma(c')}\times F_2(a,b,b',c,c' ;x,y)$$$$=\int_0^1\int_0^1s^{b-1}t^{b'-1}(1-s)^{c-b-1}(1-t)^{c'-b'-1}(1-sx-ty)^{-a}dsdt. $$
The function $z=F_2$ satisfies the system 
$$E_2(a,b,b',c,c';x,y):\quad  P_2(a,b,b',c,c';x,y)z=Q_2(a,b,b',c,c';x,y)z=0,$$
where
\begin{eqnarray*}
&& P_2(a,b,b',c,c';x,y)=D(c-1+D)-x(a+D+D')(b+D), \\[2mm]
&& Q_2(a,b,b',c,c';x,y)=D'(c'-1+D')-y(a+D+D')(b'+D). 
\end{eqnarray*}
The collection of solutions is denoted by $S_2(a,b,b',c,c';x,y)$.
\subsection{Reducibility conditions for  $E_2$ and $E_1$}
As for the reducibility of the systems $E_2$ and $E_1$, 
the following is known:
\begin{fact}\label{redF2}$(${\rm \cite{Bod}}$)$ 
Appell's system $E_2(a,b,b',c,c')$ is reducible if and only if 
at least one of
$$a,\  b,\ b',\ c-b,\ c'-b',\ c-a,\ c'-a,\  c+c'-a$$ 
is an integer.
\end{fact}
 
\begin{fact}\label{redF1}$(${\rm \cite{MS1}}$)$ 
Appell's system $E_1(a,b,b',c)$ is reducible if and only if at least one of
$$b+b'-c,\ b,\ b',\ c-a,\ a$$ is an integer.
\end{fact}

\subsection{System $E_2(a,b,b',c,c')$ under $ a=c'$}

The system $E_2(a,b,b',c,c')$ is reducible when $ a=c'$, Fact \ref{redF2}.
In fact, we see that the system $E_1(b, a-b', b',c)$ is a subsystem
of $E_2(a,b,b',c,{a})$; precisely, we have
\begin{proposition} \label{prop:s2}
$$(1-y)^{-b'}S_2\left(a,b,b',c,a;x,-\frac{y}{1-y}\right)
\supset S_1(b,a-b',b',c;x,x(1-y)).$$
\end{proposition}


We give three ``proof''s: one using power series,
Subsection \ref{subsec:power}, 
one using integral representations, Subsection \ref{subsec:integ},
and one manipulating differential equations, Subsection \ref{subsec:equat}. 
The former two are valid only under some non-integral conditions 
on parameters, which we do not give explicitly. 
Though the last one is valid for any parameters, 
it would be not easy to get a geometric meaning.

\subsubsection{Power series} \label{subsec:power}

The following fact explains the inclusion in Proposition \ref{prop:s2}.
\begin{fact} {\rm \cite[p. 80]{Bailey}. } \label{Bailey1}
$$(1-y)^{-b'}F_2\left(a,b,b',c,a;x,-\frac{y}{1-y}\right)=F_1(b,a-b',b',c;x,x(1-y)).$$
\end{fact}

\subsubsection{Integral representation} \label{subsec:integ}

We consider the integral
$$I=\int\! \int s^{b-1}t^{b'-1}(1-s)^{c-b-1}(1-t)^{c'-b'-1}(1-sx-ty)^{-a}dsdt, $$
which is a solution of the system $E_2(a,b,b',c,c';x,y)$.
We change the coordinate $t$ into $\tau$ as
$$\tau=\frac{Nt}{1-sx-yt}, \quad N=1-y-sx,$$
which sends 
$$t=0,\quad 1,\quad \frac{1-sx}y\qquad {\rm to}\qquad 
\tau=0,\quad 1,\quad \infty.$$ 
The inverse map is
$$t=\frac{1-sx}{D}\tau,\qquad D=y\tau+N.$$
Since
$$1-t=\frac{(1-\tau)N}D,\quad 1-sx-ty=\frac{(1-sx)N}D,\quad dt=\frac{(1-sx)N}{D^2}d\tau+*ds,$$
we have
$$I=\iint s^{b-1}(1-s)^{c-b-1}(1-sx)^{b'-a}N^{c'-b'-a}\cdot \tau^{b'-1}(1-\tau)^{c'-b'-1}\cdot D^{a-c'}dsd\tau.$$
This implies, if $a=c'$, then the double integral above becomes 
the product of the Beta integral
$$\int \tau^{b'-1}(1-\tau)^{c'-b'-1}d\tau$$
and the integral
\[
\begin{array}{ll}
J&=\dsp\int s^{b-1}(1-s)^{c-b-1}(1-sx)^{b'-a}N^{c'-b'-a}ds\\
&=(1-y)^{-b'}\dsp\int s^{b-1}(1-s)^{c-b-1}(1-sx)^{b'-a}
\left(1-\dsp\frac{x}{1-y}s\right)^{-b'}ds,\end{array}
\]
which is an element of the space
$(1-y)^{-b'}S_1(b,a-b',b',c;x,x/(1-y)$.
This shows 
$$S_2(a,b,b',c,a;x,y)\supset (1-y)^{-b'}
S_1\left(b,a-b',b',c;x,\frac{x}{1-y}\right),$$
which is equivalent to
$$(1-y)^{-b'}S_2\left(a,b,b',c,a;x,-\frac{y}{1-y}\right)
  \supset S_1(b,a-b',b',c;x,x(1-y)).$$

\par\medskip
The bi-rational coordinate change $(s,t)\rightarrow(s,\tau)$ is so made that the lines defining the integrand of the integral $\iint\cdots dsdt$ may become the union of vertical lines and horizontal lines in the $(s,\tau)$-space. Actual blow-up and down process is as follows (see Figure \ref{st}).  
Name the six lines in the $st$-projective plane as:
$$\ell_1: s=0,\quad\ell_2: t=0,\quad\ell_3: s=1,\quad\ell_4: t=1,
\quad\ell_5: 1-sx-ty=0,\quad\ell_6: \infty.$$
Blow up at the 4 points (shown by circles)
$$\ell_2\cap\ell_5,\quad \ell_4\cap\ell_5,\quad \ell_1\cap\ell_3\cap\ell_6=0:1:0,\quad \ell_2\cap\ell_4\cap\ell_6=1:0:0,$$
and blow-down along the proper transforms of the line $\ell_6$ 
and two lines:
$$25':\ 1-sx\quad {\rm and}\quad 45':\ N=1-y-sx;$$
these three lines are dotted.
This takes the $st$-projective plane to $\P^1(s)\times\P^1(\tau)$ .
In the figure, lines labeled $1,2,\dots$ stand for $\ell_1,\ell_2,\dots$, and the lines labeled $25,45,246$ on the right are the blow-ups of the intersection points $\ell_2\cap\ell_5,\ell_4\cap\ell_5,\ell_2\cap\ell_4\cap\ell_6$, respectively.  
The line obtained by blowing up the point $\ell_1\cap\ell_3\cap\ell_6=0:1:0$ 
is the line defined by $D=y\tau+1-y-sx$, which should be labeled by $136$.
\begin{figure}[htb] \begin{center}
\includegraphics[width=13cm]{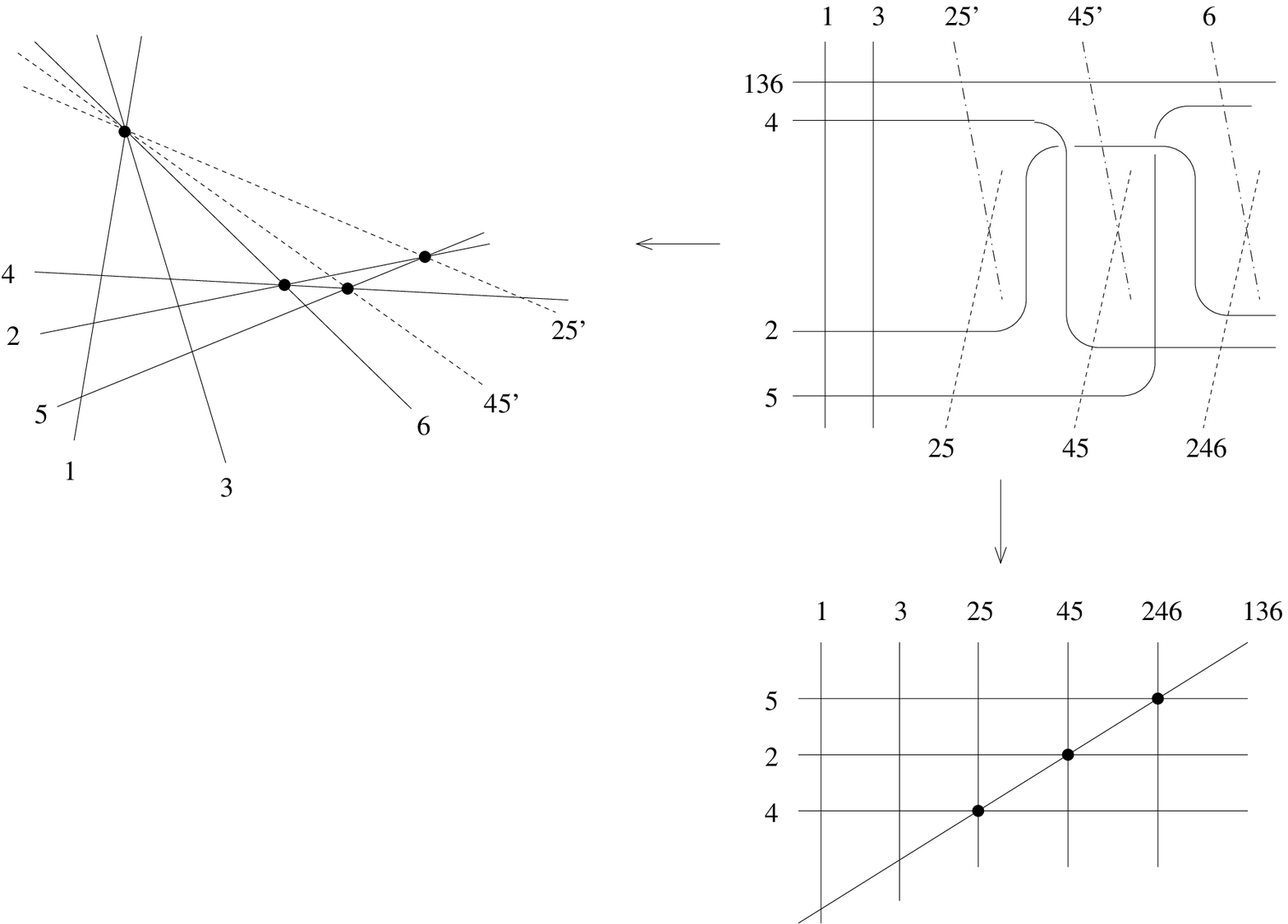}
\end{center} 
\caption{Birational map $(s,t)\to(s,\tau)$} 
\label{st} 
\end{figure}  

\subsubsection{System of differential equations} \label{subsec:equat}

A proof of the inclusion in Proposition \ref{prop:s2} that is 
valid for any parameters is done as follows.
Let $z$ be a solution of the system $E_1(a,b,b',c;x,y)$. Then, 
the system $E_1:P_1z=Q_1z=R_1z=0$ yields $\C(x,y)$-linear expressions 
of $z_{xx}, z_{xy}$ and $z_{yy}$ in terms of $z_x,z_y$ and $z$. 
Substitute these expressions into the system $E_2:P_2z=Q_2z=0$.
Then, we get two linear forms in $z_x,z_y$ and $z$. 
We now have only to see their coefficients vanish for the given parameters
after a change of coordinates and a change of the unknown by multiplying a simple factor.
We do not here present the actual computation, because if we put $x^3=0$ in the  proof of Proposition \ref{X3FD} in Subsection \ref{secondproof}, {manipulating differential equations}, it gives essentially a proof of Proposition \ref{prop:s2}.

\subsection{System $E_2(a,b,b',c,c')$ under $ a=c=c'$}
When $a=c=c'$, applying Proposition 1.3 also for $a=c$, we see that the system $E_2(a,b,b',a,a)$ has two subsystems isomorphic to $E_1$. The intersection of the two $E_1$'s would be the Gauss hypergeometric equation.
In fact, we have the following proposition.
\begin{proposition} \label{prop:s2g}
$$S_2(a,b,b',a,a;x,y)
\supset(1-x)^{-b}(1-y)^{-b'}S\left(b,b',a;\frac{xy}{(1-x)(1-y)}\right).$$
\end{proposition}
Similar to the argument of the previous Subsection,
we can give three ``proof''s: one using power series,
one using integral representations, 
and one manipulating differential equations.
We give a sketch of them in the following.

\subsubsection{Power series}

The following identity explains the inclusion above.
\begin{fact} {\rm \cite[p. 81]{Bailey}. } \label{Bailey2}
$$F_2(a,b,b',a,a;x,y)=(1-x)^{-b}(1-y)^{-b'}
F\left(b,b',a;\frac{xy}{(1-x)(1-y)}\right).$$
\end{fact}

\subsubsection{Integral representation}
We continue the argument in \S \ref{subsec:integ}.
In the integral $J$ in \S \ref{subsec:integ} above, 
change the coordinate from $s$ to $\sigma$ as
$$s=\frac{\sigma}M,\quad \sigma=\frac{(1-x)s}{1-xs},\qquad M=x\sigma+1-x$$
sending 
$$s=0,\ 1,\ \frac1x\quad{\rm to}\quad \sigma=0,\ 1,\ \infty.$$
Since 
$$1-s=\frac{(1-x)(1-\sigma)}M,\quad 1-sx=\frac{1-x}M,\quad 1-\frac{xs}{1-y}=\frac{(1-x)(1-y)-xy\sigma}{(1-y)M},$$
we have
$$J=(1-x)^{c-b-a}(1-y)^{-b'}\int \sigma^{b-1}(1-\sigma)^{c-b-1}\left(1-\frac{xy\sigma}{(1-x)(1-y)}\right)^{-b'}M^{a-c}d\sigma.$$
This implies, if $a=c$, then 
$$J=(1-x)^{-b}(1-y)^{-b'}\int \sigma^{b-1}(1-\sigma)^{a-b-1}\left(1-\frac{xy\sigma}{(1-x)(1-y)}\right)^{-b'}d\sigma.$$
This shows 
$$(1-y)^{-b'}S_1\left(b,a-b',b',a;x,\frac{x}{1-y}\right)\supset
(1-x)^{-b}(1-y)^{-b'}S\left(b,b',a;\frac{xy}{(1-x)(1-y)}\right),$$
which of course implies the inclusion relation in Proposition 
\ref{prop:s2g} by combination with that in Proposition \ref{prop:s2}.

\subsubsection{System of differential equations}
Put
$$z=(1-x)^{-b}(1-y)^{-b'}u(t),\quad t=xy/(1-x)/(1-y).$$
We have
$$z_x=(1-x)^{-b}(1-y)^{-b'}u't_x+\cdots,\quad z_{xx}=(1-x)^{-b}(1-y)^{-b'}u''(t_x)^2+\cdots,$$
and so on. 
Assume that $u(t)\in S(b,b',a;t)$. The equation $P(b,b',a;t)u(t)=0$ gives a linear expression of $u''$ in $u'$ and $u$. Substitute these expressions in  
$$P_2(a,b,b',a,a;x,y)z=x(1-x)z_{xx}-xyz_{xy}+\cdots$$
and we get the product of $(1-x)^{-b}(1-y)^{-b'}$ 
and a $\C(x,y)$-linear combination of $u$ and $u'$. 
The coefficients vanish if $a=c$. 
If we do the same for $Q_2(a,b,b',a,a;x,y)z$, then we find that it vanishes when $a=c'$. 

\section{Solutions expressed as indefinite integrals}\label{indefinite0}

We show that some indefinite integrals solve the system $E_2(a,b,b',a,a)$.
We begin with some well-known facts.

\begin{lemma}\label{wellknown1}
$$P(a,b,c;x)\Phi=b x\frac{\partial}{\partial s}
\left(\frac{s(1-s)}{x-s}\Phi\right),$$
where 
\[
P(a,b,c;x)=D(D+c-1)-x(D+a)(D+b),
\quad\Phi=s^{b-c}(1-s)^{c-a-1}(x-s)^{-b}.
\]
\end{lemma}

\begin{proof} Note that
$$\left(D+D_s+\frac{c\!-\!a\!-\!1}{1-s}\right)
  \Phi=(-a-1)\Phi, \quad D=x\frac{\partial}{\partial x},\quad D_s=s\frac{\partial}{\partial s}.$$
This implies
\begin{eqnarray*}
&(D +a)\Phi=-\dsp \left(D_s+\frac{c -a-1}{1-s}+1\right)\Phi,& \\
& (D +c-1)\Phi=-\dsp \left(D_s+\frac{(c-a-1)s}{1-s}+1\right)\Phi.&
\end{eqnarray*}
Since
$$D \Phi=\frac{-b x}{x-s}\Phi,\quad 
(D+b)\Phi=-\frac{b s}{x-s}\Phi,\quad D_s+1=\frac{\partial}{\partial s}s,$$
we have 
\[
\begin{array}{ll}
P\Phi&=\dsp\left\{\left(D_s+\frac{(c\!-\!a\!-\!1)s}{1-s}+1\right)
    \frac{b x}{x-s}
    -z\left(D_s+\frac{c\!-\!a\!-\!1}{1-s}+1\right)
   \frac{b s}{x-s}\right\}\Phi\\[4mm]
  &=\dsp b x(D_s+1)\frac{1-s}{x-s}\Phi \\[4mm]
  &=\dsp b x\frac{\partial}{\partial s}\left(\frac{s(1-s)}{x-s}\Phi\right).
\end{array}
\]
\end{proof} 

\begin{lemma}\label{wellknown2} The indefinite integral 
$$u=\int_p^s\Phi ds,\quad\Phi=s^{b-c}(1-s)^{c-a-1}(t-s)^{-b},\quad p\in\{0,1,t,\infty\}$$
solves $E_1(a,0,b,c:s,t)$. In particular, $S_1(a,0,b,c;s,t)\supset S(a,b,c;t)$.
\end{lemma}
\begin{proof}
Since $u_s=\Phi$, we have
$$ u_{ss}=\left(\frac{b-c}{s}-\frac{c-a-1}{1-s}-\frac{-b}{t-s}\right)u_s,\quad u_{st}=\frac{-b}{t-s}u_s.$$
Lemma \ref{wellknown1} leads to 
$$P(a,b,c;t)u=bt\frac{s(1-s)}{t-s}u_s.$$
Let $P_1,Q_1$ and $R_1$ be the operators generating the system $E_1(a,0,b,c:s,t)$:
$$\begin{array}{ll}
P_1(a,0,b,c;s,t)&
 =s(1-s)\partial_{ss}+t(1-s)\partial_{st}+(c-(a+1)s)\partial_s,\\[2mm]
Q_1(a,0,b,c;s,t)&
 =P(a,b,c;t)/t+s(1-t)\partial_{st}-bs\partial_s,\\[2mm]
R_1(a,0,b,c;s,t)&=(s-t)\partial_{st}-b\partial_s;\end{array}$$
refer to Section 1.
Note that $P(a,b,c;t)/t=t(1-t)\partial_{tt}+\cdots$. 
By using the above identities, we have
$$R_1u=(s-t)u_{st}-bu_s=0,$$
and
$$P_1u=s(1-s)\left(\frac{b-c}{s}-\frac{c-a-1}{1-s}-\frac{-b}{t-s}\right)u_s+t(1-s)\frac{bu_s}{s-t}+(c-(a+1)s)u_s=0,$$
$$Q_1u=bt\frac{s(1-s)}{t-s}u_s+\left(\frac{s(1-t)}{s-t}-bs\right)u_s=0.$$
Furthermore, for $z\in S(a,b,c;t)$, the forms of operators above imply
that $z$ lies in $S_1(a,0,b,c;s,t)$.
\end{proof} 

We now use the following fact:

\begin{fact} {\rm \cite[p. 78]{Bailey}. } \label{Bailey3}
$$F_1(a,b,b',c;x,y)=(1-x)^{-a}F_1\left(a,c-b-b',b',c;\frac{-x}{1-x},\frac{y-x}{1-x}\right).$$\end{fact}

\noindent From this fact we get, when $c=b+b'$,  
$$S_1(a,b,b',b+b';x,y)=(1-x)^{-a}S_1\left(a,0,b',b+b';
   \frac{-x}{1-x},\frac{y-x}{1-x}\right).$$
If we put 
$$y=\frac{x}{1-\eta}, $$ 
then
$$\frac{y-x}{1-x}=\frac{x\eta}{(1-x)(1-\eta)}.$$
Thus we have
$$\begin{array}{ll}
S_2(a,b,b',a,a;x,y)&\supset(1-y)^{-b'}S_1\left(b,a-b',b',a;x,\dsp\frac{x}{1-y}\right)\\[3mm]
&=(1-y)^{-b'}(1-x)^{-b}S_1\left(b,0,b',a;\dsp\frac{-x}{1-x},\frac{xy}{(1-x)(1-y)}\right)\\[3mm]
&\supset (1-y)^{-b'}(1-x)^{-b}
S\left(b,b',a;\dsp\frac{xy}{(1-x)(1-y)}\right).
\end{array}$$
This agrees with the inclusion in Proposition \ref{prop:s2g}. 
In particular, by the inclusion 
$$
(1-y)^{b'}(1-x)^{b}S_2(a,b,b',a,a;x,y)
\supset
S_1\left(b,0,b',a;\dsp\frac{-x}{1-x},\frac{xy}{(1-x)(1-y)}\right)
$$
and Lemma \ref{wellknown2} 
we get solution of $(1-y)^{b'}(1-x)^{b}E_2(a,b,b',a,a;x,y)$ 
represented by the indefinite integral:
$$f_1:=\int_0^\xx s^{b'-a}(1-s)^{a-b-1}(\xx\yy-s)^{-b'}ds,
\qquad\xx=\frac{-x}{1-x},\  \yy=\frac{-y}{1-y}.$$
Starting point of the path of integration can be any point $p\in\{0,1,t,\infty\}$, so we choose $p=0$, just for simplicity. 
By exchanging the role of $x$ and $y$, we get an inclusion
$$
(1-y)^{b'}(1-x)^{b}S_2(a,b,b',a,a;x,y)
\supset
S_1\left(b',b,0,a;\frac{xy}{(1-x)(1-y)},\dsp\frac{-y}{1-y},\right)
$$
and another solution of $(1-y)^{b'}(1-x)^{b}E_2(a,b,b',a,a;x,y)$
represented by the indefinite integral:
$$
\int_0^\yy s^{b-a}(1-s)^{a-b'-1}(\xx\yy-s)^{-b}ds.
$$
After the change $s\to \xx\yy/s$, it can be also expressed as 
$$f_2:=(\xx\yy)^{-a+1}\int_0^\xx s^{b'-1}(1-s)^{-b}(\xx\yy-s)^{a-b'-1}ds.$$
Thus, we have:

\begin{theorem}
\label{matome}
We have the following inclusions of the spaces of solutions:
$$
\begin{matrix}
(1-y)^{b'}(1-x)^{b}S_2(a,b,b',a,a;x,y)
& \supset &
S_1\left(b,0,b',a;\dsp\frac{-x}{1-x},\frac{xy}{(1-x)(1-y)}\right)
\\
\cup & & \cup
\\
S_1\left(b',b,0,a;\frac{xy}{(1-x)(1-y)},\dsp\frac{-y}{1-y},\right) 
& \supset &
S\left(b,b',a;\dsp\frac{xy}{(1-x)(1-y)}\right)
\end{matrix}
$$

Moreover, the collection of solutions
$(1-y)^{b'}(1-x)^{b}S_2(a,b,b',a,a;x,y)$ is spanned by $f_1,f_2$ and  $S\left(b,b',a;\dsp\frac{xy}{(1-x)(1-y)}\right).$\end{theorem}
{This will play a key role to understand the Schwarz map of a system $E_2$ with specific parameters, which will be introduced in \S \ref{E36}.}

\section{Monodromy representation of $E_2$}
\label{sec:monod}
In this section, we  {study the} monodromy representation of 
$E_2=E_2(a,b,b',c,c')$.
Though {it is assumed} in \cite{MY} that the parameters 
satisfy the irreducibility condition in Fact 1.1:
$$a,\ b,\ b',\ c-b,\ c'-b',\ a-c-c',\quad a-c,\ a-c'\notin \Z,$$
in this section, we only assume the weaker condition
\begin{equation}
\label{eq:semi-non-int}
a,\ b,\ b',\ c-b,\ c'-b',\ a-c-c'\notin \Z.
\end{equation}
We modify Theorem 7.1 in \cite{MY} so that the {statements remain valid for these parameters.} We will apply the result of this section in \S \ref{special}.
\subsection{Twisted homology groups and the intersection form}{Let} 
$$X=\{(x,y)\in \C^2\mid x(x-1)y(y-1)(x+y-1)\ne0\}$$ 
be the complement of the singular locus 
of $E_2$. 
For each $(x,y)\in X$, we consider  a multi-valued function  
$$
\psi=t_1^{b-1}(1-t_1)^{c-b-1}t_2^{b'-1}(1-t_2)^{c'-b'-1}(1-t_1x-t_2y)^{-a},
$$
on  
$$T_{x,y}=\{(t_1,t_2)\in \C^2\mid 
t_1(1-t_1)t_2(1-t_2)(1-t_1x-t_2y)
\ne0\}.$$ 
\newcommand{\T}{T}
As in \S1.6 and \S1.7 of Chapter 2 in \cite{AK}, we define   
the twisted homology group $H_2(T_{x,y},\psi)$ associated with $\psi$ 
and locally finite one $\Hl_2(T_{x,y},\psi)$. {Under some genericity condition, the integral of $\psi$ over a twisted cycle gives a solution of $E_2$.}

If $a-c,\ a-c'\notin \Z$ then 
the natural map $\imath:H_2(T_{x,y},\psi)\to\Hl_2(T_{x,y},\psi)$ 
is bijective, {and the}  inverse map
$\mathrm{reg}:\Hl_2(T_{x,y},\psi)\to H_2(T_{x,y},\psi)$ 
is called the regularization. 
{In general,} the map $\imath:H_2(T_{x,y},\psi) \to\Hl_2(T_{x,y},\psi)$ is neither injective 
nor surjective, {however we still have}
the isomorphism
$$\mathrm{reg}:\im(\imath)\to H_2(T_{x,y},\psi)/\ker(\imath).$$

{Under condition (3.1), thanks to} the vanishing theorem of cohomology groups in \cite{C},  rank of $\Hl_2(T_{x,y},\psi)$ and $H_2(T_{x,y},\psi)$ are
equal to the Euler number $\chi(T_{x,y})=4$ of $T$, and the bilinear form --  {\it the intersection form} -- 
$$\fI: \Hl_2(T_{x,y},\psi)\times H_2(T_{x,y},\psi^{-1})\to \C$$  
is non-degenerate.

\subsection{Monodromy representation}{
Let $U$ be a small simply connected domain in $X$.
We can identify the local solution space to $E_2$ 
on $U$ with the trivial vector bundle 
$\bigcup_{(x,y)\in U} \Hl_2(T_{x,y},\psi)$ via the Euler type 
integral representation of solutions to $E_2$.
The monodromy representation of 
$E_2$ is equivalent to that of the local system 
$$
\mathcal{H}_2^{\mathrm{lf}}(\psi)= 
\bigcup_{(x,y)\in X} \Hl_2(T_{x,y},\psi)$$
over $X$.
We also consider a local system 
$$\mathcal{H}_2(\psi^{-1})=
\bigcup_{(x,y)\in X} H_2(T_{x,y},\psi^{-1})$$
over $X$. 
We fix a small positive real number $\e$, and let $(\e,\e)$ a base point in $X$. Denote the germs at this point of the local systems 
$\mathcal{H}_2^{\mathrm{lf}}(\psi)$ and $\mathcal{H}_2(\psi^{-1})$
by 
$$\Hl_2(T,\psi),\quad H_2(T,\psi^{-1}),\qquad T=T_{\e,\e},$$
respectively. 
Let
\begin{eqnarray*}
\CM^\mu&:&\pi_1(X,(\e,\e))\ni \rho \mapsto 
\CM^\mu_\rho\in \mathrm{GL}(\Hl_2(T,\psi)),\\
\cCM^{-\mu}&:&\pi_1(X,(\e,\e))\ni \rho \mapsto 
\cCM^{-\mu}_\rho\in \mathrm{GL}(H_2(T,\psi^{-1}))
\end{eqnarray*}
be the monodromy representations of 
$\mathcal{H}_2^{\mathrm{lf}}(\psi)$ and $\mathcal{H}_2(\psi^{-1})$
with respect to $(\e,\e)$.
$\CM^\mu_\rho$ and $\cCM^{-\mu}_\rho$ are 
called the circuit transformations along $\rho$.
\begin{proposition}
\label{prop:inv-subspace}
\begin{itemize}
\item[$(1)$] 
The image $\im(\imath)$ of the natural map 
$\imath:H_2(T,\psi)\to \Hl_2(T,\psi)$ 
is invariant under the monodromy representation $\CM^\mu$. 
\item[$(2)$] 
The kernel $\ker(\imath')$ of the natural map 
$\imath':H_2(T,\psi^{-1})\to \Hl_2(T,\psi^{-1})$ 
is invariant under the monodromy representation $\cCM^{-\mu}$. 
\end{itemize}
\end{proposition}
\begin{proof}
It is clear that $H_2(T,\psi)$ and $H_2(T,\psi^{-1})$ are invariant 
under $\CM^\mu$ and $\cCM^{-\mu}$, respectively. 
We have only to note that 
the natural maps $\imath$ and $\imath'$ commute with $\CM^\mu$ and $\cCM^{-\mu}$.
\end{proof}
\begin{remark}
We will see that if $a-c\in \Z$ or $a-c'\in \Z$, then both of
$\im(\imath)$ and $\ker(\imath')$ are 
proper subspaces. 
Thus monodromy representations $\CM^\mu$ and $\cCM^{-\mu}$ 
are reducible in this case.
\end{remark}
}

\begin{lemma}
\label{lem:duality}
\begin{enumerate}
\item[$(1)$] Let $\Delta$ and $\cD$
be elements of 
$\Hl_2(T,\psi)$ and $H_2(T,\psi^{-1})$, respectively. Then we have
$$\fI(\CM_\rho^\mu(\Delta),\cCM_\rho^{-\mu}(\cD))=
\fI(\Delta,\cD).
$$
\item[$(2)$] 
Suppose that $W=\Hl_2(T,\psi)$ is decomposed into the direct sum 
of the eigenspaces $W_1,\dots, W_r$ of $\CM_\rho^\mu$ 
of eigenvalues $\lambda_1,\dots,\lambda_r$. Then 
$\widetilde W=H_2(T,\psi^{-1})$ is decomposed into the direct sum 
of the eigenspaces $\widetilde W_1,\dots, \widetilde W_r$ of 
$\cCM_\rho^{-\mu}$ of eigenvalues $\lambda_1^{-1},\dots,\lambda_r^{-1}$.
The eigenspace $\widetilde W_i$ is characterized as
$$\widetilde W_i=\bigcap_{1\le j\le r}^{j\ne i} W_j^\perp,\quad 
W_j^\perp=\{\widetilde w \in \widetilde W
\mid \fI(w_j,\widetilde w)=0\ \textrm{ for any } w_j\in W_j\}.$$
\end{enumerate}
\end{lemma}
\begin{proof}
(1) {The intersection number $\fI(\Delta,\cD)$ is stable under small deformations of $\Delta$ and $\cD$.}

\medskip\noindent
(2) Since $\CM_\rho^\mu$ belongs to the general linear group, we have 
$\lambda_1\cdots \lambda_r\ne0$.
Let $w_j$ be any element of $W_j$ $(j=2,\dots,r)$ and 
let $\widetilde w$ be any element of 
$W'_1=\bigcap\limits_{2\le j\le r} W_j^\perp$.
Since we have 
$$
\lambda_j\fI(w_j,\cCM_\rho^{-\mu}(\widetilde w))
=\fI(\lambda_jw_j,\cCM_\rho^{-\mu}(\widetilde w))
=\fI(\CM_\rho^\mu(w_j),\cCM_\rho^{-\mu}(\widetilde w))
=\fI(w_j,\widetilde w)=0$$
for $j=2,\dots,r$, $\cCM_\rho^{-\mu}(\widetilde w)$ belongs to 
the space $W'_1$. 
Thus $\cCM_\rho^{-\mu}$ induces a linear transformation of $W'_1$.
Let $\lambda'$ be an eigenvalue of the restriction of $\cCM_\rho^{-\mu}$ 
to $W'_1$, and let $w'_1\in W'_1$ be an eigenvector of $\lambda'$.
Since the intersection from $\fI$ is non degenerate, there exists 
$w_1\in W_1$ 
such that $\fI(w_1,w'_1)\ne 0$. Note that 
$$0\ne \fI(w_1,w'_1)=\fI(\CM_\rho^{\mu}(w_1),\cCM_\rho^{-\mu}(w'_1))
=\fI(\lambda_1w_1,\lambda'w'_1)=(\lambda_1\lambda')\cdot \fI(w_1,w'_1).
$$
Hence we have $\lambda_1\lambda'=1$, i.e., $\lambda'=\lambda_1^{-1}$ and 
$W'_1\subset \widetilde W_1$. Similarly, we have
$W'_i\subset \widetilde W_i$ for $i=2,\dots,r$. 
To show $W_i'= \widetilde W_i$, 
we consider the restriction of $\cCM_\rho^{-\mu}$ to 
$W''_1=W'_1\cap \{w_1\}^\perp$,   where 
$$\{w_1\}^\perp=\{\widetilde w\in \widetilde W\mid \fI(w_1,\widetilde w)=0\}.$$
Take its eigenvector ${\widetilde w}'$ and repeat the above argument. 
In this way, we have 
$$\dim W_1'=\dim W_1\le \dim  \widetilde W_1'.$$
Since 
$$\dim W= \sum_{i=1}^r \dim W_i\le 
\sum_{i=1}^r\dim \widetilde W_i\le \dim \widetilde W,\quad 
\dim W=\dim \widetilde W,
$$
we have  $\dim W_i'=\dim W_i= \widetilde W_i$ for $i=1,\dots,r$. 
\end{proof}

\subsection{Twisted cycles}
Let $\square_j$ $(i=1,\dots,6)$ be locally finite chains shown in Figure 
\ref{fig:basis}. 
We specify a branch of $\psi$ on each chain by the assignment of 
$\arg(f_j)$ on it 
as in Table \ref{tab:arg},
where 
$$f_1=t_1,\quad f_2=1-t_1,\quad 
 f_3=t_2,\quad f_4=1-t_2,\quad f_5=1-t_1x-t_2y,$$
{and load $\psi$ to get the 
 locally finite twisted cycles $\square_i^\psi\in\Hl_2(T,\psi)$.}
\begin{figure}[htb]
\begin{center}
\includegraphics[width=10cm]{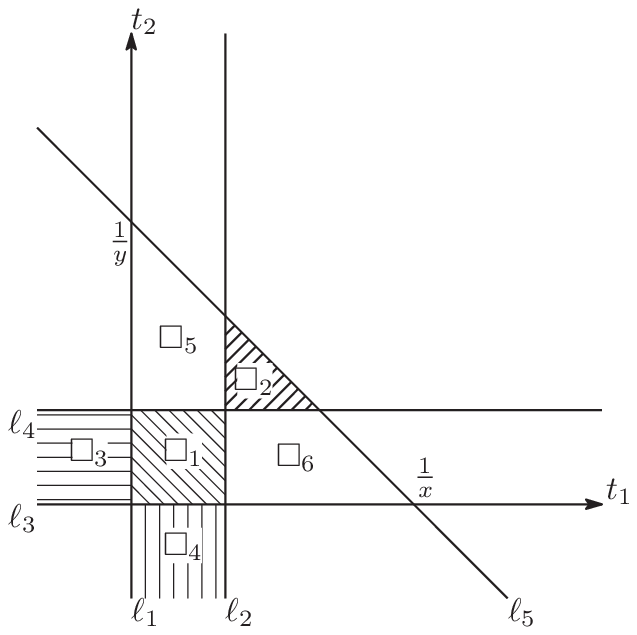}
\end{center}
\caption{$2$-cycles}
\label{fig:basis}
\end{figure}
\begin{figure}[htb]
\begin{center}
\includegraphics[width=12cm]{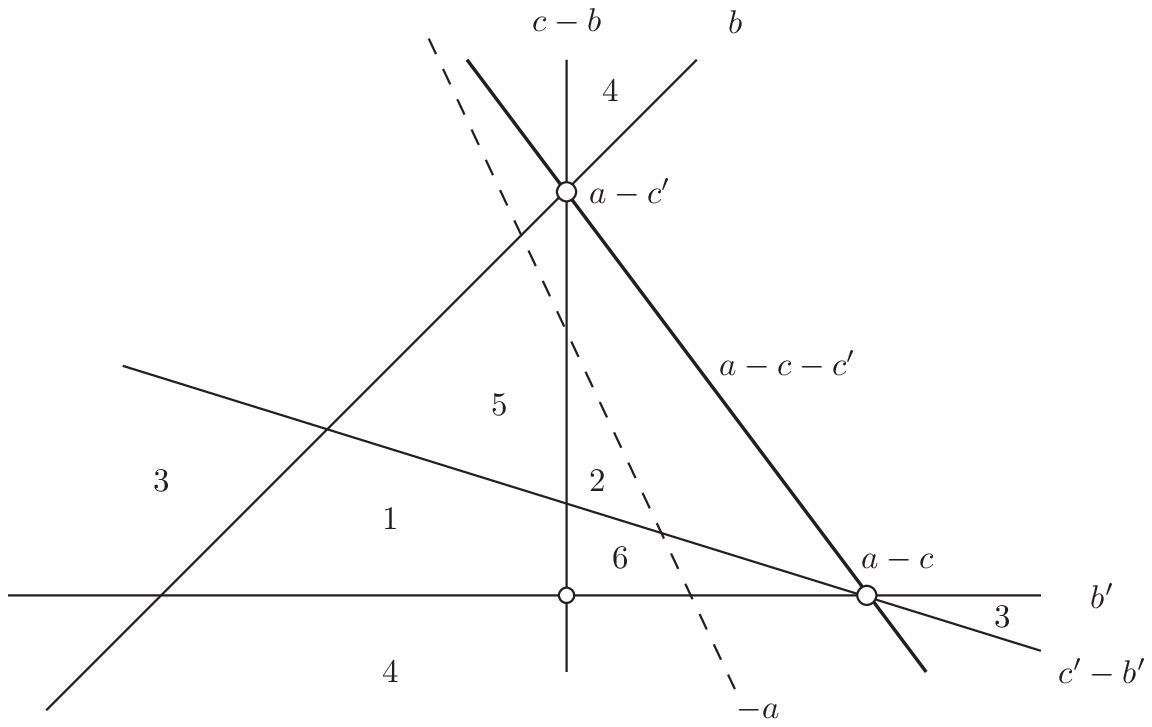}
\end{center}
\caption{Exponents around the line at infinity}
\label{P2fig}
\end{figure}
\begin{table}[htb]
  \centering
  \begin{tabular}[htb]{|c|c|c|c|c|c|}
\hline
 & $f_1=t_1$ & $f_2=1-t_1$ & $f_3=t_2$ & $f_4=1-t_2$ & 
$f_5=1-t_1x-t_2y$\\
\hline
$\square_1$ &  $0$  &   $0$   & $0$   & $0$     & $0$   \\
$\square_2$ &  $0$  &   $\pi$ & $0$   & $\pi$   & $0$   \\
$\square_3$ &$-\pi$ &   $0$   & $0$   & $0$     & $0$   \\
$\square_4$ &   $0$ &   $0$   &$-\pi$ & $0$     & $0$   \\
\hline
$\square_5$ &  $0$  &   $0$   & $0$   & $\pi$     & $0$   \\
$\square_6$ &  $0$  &   $\pi$ & $0$   & $0$   & $0$   \\
\hline
  \end{tabular}
  \caption{List of $\arg(f_j)$ on $\square_i$}
  \label{tab:arg} 
\end{table}
It {will be shown} that 
$\Delta_1=\square_1^\psi,\dots, \Delta_4=\square_4^\psi$ form a basis 
of $\Hl_2(T,\psi)$ in Corollary \ref{cor:basis}.

{We choose} elements in $H_2(T,\psi^{-1})$ as 
$${\cD_1}=\mathrm{reg}(\square_1^{\psi^{-1}}),\quad 
{\cD_2}=\mathrm{reg}(\square_2^{\psi^{-1}}),
$$
$$
{\cD_3}=(\mu_{125}-1)
\mathrm{reg}(\square_3^{\psi^{-1}}),\quad  
{\cD_4}=(\mu_{345}-1)
\mathrm{reg}(\square_4^{\psi^{-1}}),$$  
where {$\mu_{ij\cdots}=\mu_i\mu_j\cdots$} and  
$$\mu_1=e^{2\pi\sqrt{-1}b},\ \mu_2=e^{2\pi\sqrt{-1}(c-b)},\ 
\mu_3=e^{2\pi\sqrt{-1}b'},\ \mu_4=e^{2\pi\sqrt{-1}(c'-b')},\ 
\mu_5=e^{-2\pi\sqrt{-1}a}.$$
{Note that in terms of $\mu_j$'s the irreducible condition in Fact 1.1 is
$$\mu_i(i=1,\dots,5),\ \mu_{12345},\quad \mu_{125},\ \mu_{345}\not=1,$$
the condition (3.1) is
$$\mu_i(i=1,\dots,5),\ \mu_{12345}\not=1,$$
and (see Figure \ref{P2fig})
$$c-a\in \Z\Leftrightarrow \mu_{125}=1,\quad c'-a\in \Z\Leftrightarrow \mu_{345}=1,$$
and that $(a,b,c)=(4/3,\mathbf{2/3},\mathbf{4/3})$ satisfies (3.1) and $\mu_{ijk}=1$ for any $i,j,k$.}
Explicitly, ${\cD_3}$ and ${\cD_4}$ can be written as:
$$(\mu_{125}-1)\Big[\big(\frac{\circlearrowleft_\infty^+}{\mu_{125}-1}
+[\frac{-1}{\delta},-\delta]-
\frac{\circlearrowleft_{0}^-}{\mu_1^{-1}-1}\big)\times 
\big(\frac{\circlearrowleft_0^+}{\mu_3^{-1}-1}
+[\delta,1-\delta]-
\frac{\circlearrowleft_{1}^-}{\mu_4^{-1}-1}\big)\Big]^{\psi^{-1}}, 
$$
$$(\mu_{345}-1)\Big[
\big(\frac{\circlearrowleft_0^+}{\mu_1^{-1}-1}
+[\delta,1-\delta]-
\frac{\circlearrowleft_{1}^-}{\mu_2^{-1}-1}\big)
\times 
\big(\frac{\circlearrowleft_\infty^+}{\mu_{345}-1}
+[\frac{-1}{\delta},-\delta]-
\frac{\circlearrowleft_{0}^-}{\mu_3^{-1}-1}\big)\Big]^{\psi^{-1}}, 
$$
where $\delta$ is a small positive real number, 
$[\frac{-1}{\delta},-\delta]$ and $[\delta,1-\delta]$ are 
closed intervals, 
$\circlearrowleft_\infty^+$ is the negatively oriented circle 
of which radius, center and terminal are $1/\delta$, $0$ and $-1/\delta$,  
$\circlearrowleft_q^\pm$ $(q=0,1)$ is the positively oriented circle of 
which radius, center and terminal are $\delta$, $q$ and $q\pm \delta$.

{Notice that the definition of the twisted cycles ${\cD_i}$ 
$(i=1,\dots,4)$ make sense} even in the case $a-c\in \Z$ or $a-c'\in \Z$.
{Indeed, this specialization gives no harm to  $\cD_i$ $(i=1,2)$, and thanks to the above expression, when $\mu_{125}=1$ and $\mu_{345}=1$,} we have

\begin{eqnarray*}
{\cD_3}&=&
\Big[\circlearrowleft_\infty^+\times 
\big(\frac{\circlearrowleft_0^+}{\mu_3^{-1}-1}
+[\delta,1-\delta]-
\frac{\circlearrowleft_{1}^-}{\mu_4^{-1}-1}\big)\Big]^{\psi^{-1}}, \\
{\cD_4}&=&
\Big[
\big(\frac{\circlearrowleft_0^+}{\mu_1^{-1}-1}
+[\delta,1-\delta]-
\frac{\circlearrowleft_{1}^-}{\mu_2^{-1}-1}\big)
\times 
\circlearrowleft_\infty^+\Big]^{\psi^{-1}},
\end{eqnarray*}
respectively.

\begin{remark}  
\label{rem:kernel}
Suppose that $a-c,\ a-c'\in \Z$.  
Then the twisted cycles ${\imath(}{\cD_3})$ 
and ${\imath(}{\cD_4})$ are homologous to $0$ 
{in} $\Hl_2(T,\psi^{-1})$, since 
they are the boundary  of locally finite $3$-chains 
given by the replacement $\circlearrowleft_\infty^+\to 
\odot_\infty^+$ in their expressions, where 
$\odot_\infty^+$ is the annulus $\{t\in \C\mid 1/\delta\le |t|\}$.
They belong to $\ker(\imath')$.
{By Proposition \ref{prop:int-mat}, 
it turns out that $\im(\imath)$ is spanned by $\De_1$ and $\De_2$.}
\end{remark}

\subsection{Intersection matrices}
\begin{proposition}
\label{prop:int-mat}
The intersection matrix 
$H^\mu=\big(\fI(\Delta_i,{\cD_j})\big)_{1\le i,j,\le4}$ 
for $\Delta_1$,$\dots$,$\Delta_4\in \Hl_2(T,\psi)$
and ${\cD}_1$,$\dots$,${\cD}_4\in 
H_2(T,\psi^{-1})$ is given by
$$\left(
\begin {array}{cccc} 
\frac{(\mu_{12}-1 )(\mu_{34}-1)}
{({\mu_1}-1)({\mu_2}-1)({\mu_3}-1)({\mu_4}-1)}
&\frac {1}{({\mu_2}-1 ) ({\mu_4}-1 )}
&-\frac { {\mu_1} (\mu_{34}-1 )(\mu_{125}-1 )}
{({\mu_1}-1)({\mu_3}-1)({\mu_4}-1 ) }
&-\frac {{\mu_3}(\mu_{12}-1)(\mu_{345}-1 )}
{({\mu_1}-1 )({\mu_2}-1 )({\mu_3}-1)}
\\ 
\noalign{\medskip}{\frac {\mu_{24}}{({\mu_2}-1 )({\mu_4}-1 )}}
&{\frac {\mu_{245}-1}{ ( {\mu_2}-1 ) ( {\mu_4}-1 )( {\mu_5}-1 ) }}
&0&0
\\ 
\noalign{\medskip}-{\frac {\mu_{34}-1}{ ( {\mu_1}-1 )  
( {\mu_3}-1 )  ( {\mu_4}-1 ) }}
&0
&\frac{{\mu_1}(\mu_{34}-1)(\mu_{25}-1)}
{({\mu_1}-1)({\mu_3}-1 )({\mu_4}-1 )}
&{\frac {{\mu_3}(\mu_{345}-1 ) }{ ( {\mu_1}-1 )({\mu_3}-1 ) }}
\\ 
\noalign{\medskip}
-{\frac {\mu_{12}-1}{ ( {\mu_1}-1 )({\mu_2}-1 )({\mu_3}-1 ) }}
&0
&{\frac {{\mu_1}(\mu_{125}-1)}{({\mu_1}-1 )({\mu_3}-1 ) }}
&{\frac {{\mu_3}(\mu_{12}-1)(\mu_{45}-1 )}
{ ( {\mu_1}-1 )  ( {\mu_2}-1 )  ( {\mu_3}-1 ) }}
\end {array} \right).
$$
Its determinant is 
$$
\frac{\mu_{1234}(\mu_5-1)(\mu_{12345}-1)}
{(\mu_1-1)^2(\mu_2-1)^2(\mu_3-1)^2(\mu_4-1)^2},
$$
which does not vanish under the assumption $(\ref{eq:semi-non-int})$.
\end{proposition}
\begin{proof}
Follow \S3 of Chapter VIII in \cite{Yo} for the computation of 
the intersection numbers. 
By a straightforward calculation, we have its determinant. 
\end{proof}

Proposition \ref{prop:int-mat} yields the following corollary. 
\begin{corollary}
\label{cor:basis}
The twisted cycles 
$\Delta_1$,$\dots$,$\Delta_4$ 
and ${\cD}_1$,$\dots$,${\cD}_4$
form a basis of $\Hl_2(T,\psi)$ and 
that of $H_2(T,\psi^{-1})$, respectively.  
\end{corollary}

We express the twisted cycles $\square_5^\psi$ and $\square_6^\psi$ 
as linear combinations of $\Delta_i$.

\begin{lemma}
\label{lem:squares}
We have 
\begin{eqnarray*}
\square_5^\psi&=&-\frac{\mu_4 \mu_5-1}{\mu_5-1}\Delta_1
-\frac{\mu_{345} -1}{\mu_5-1}\Delta_4,\\
\square_6^\psi&=&-\frac{\mu_2 \mu_5-1}{\mu_5-1}\Delta_1
-\frac{\mu_{125}-1}{\mu_5-1 }\Delta_3.
\end{eqnarray*}

\end{lemma}
\begin{proof}
Set 
$$\square_5^\psi=\sum_{i=1}^4 \g_i\Delta_i=(\g_1,\dots,\g_4)
\;^t(\Delta_1,\dots,\Delta_4),$$
and compute the intersection numbers 
$\fI(\square_5^\psi,{\cD}_i)$. 
Then we have
\begin{eqnarray*}
& &(\g_1,\g_2,\g_3,\g_4)H^\mu=(\fI(\square_5^\psi,{\cD}_1),\dots,
\fI(\square_5^\psi,{\cD}_4))\\
&=&\Big(\frac{-\mu_4(\mu_1\mu_2-1)}{(\mu_1-1)(\mu_2-1)(\mu_4-1)}, 
\frac{-(\mu_4\mu_5-1)}{(\mu_2-1)(\mu_4-1)(\mu_5-1)},
\frac{\mu_1\mu_4(\mu_{125}-1)}{(\mu_1-1)(\mu_4-1)},
0\Big),
\end{eqnarray*}
which yields the expression of $\square_5^\psi$.

Since
\begin{eqnarray*}
& &(\fI(\square_6^\psi,{\cD}_1),\dots,
\fI(\square_6^\psi,{\cD}_4))\\
&=&\Big(\frac{-\mu_2(\mu_3\mu_4-1)}{(\mu_2-1)(\mu_3-1)(\mu_4-1)}, 
\frac{-(\mu_2\mu_5-1)}{(\mu_2-1)(\mu_4-1)(\mu_5-1)},0,
\frac{\mu_2\mu_3(\mu_{345}-1)}{(\mu_2-1)(\mu_3-1)}\Big),
\end{eqnarray*}
we have the expression of $\square_6^\psi$. 
\end{proof}
\begin{remark}
In \cite{MY}, we take a basis $\Delta_1$, $\Delta_2$, 
$\square_5^\psi$ and $\square_6^\psi$ of $\Hl_2(T,\psi)$. 
If $a-c,a-c'\in \Z$ then each of $\square_5^\psi$ and $\square_6^\psi$ is 
a scalar multiple of $\Delta_1$ by Lemma \ref{lem:squares}. 
\end{remark}

Similar to Lemma \ref{lem:squares}, we have the following.
\begin{lemma}
\label{lem:dual-squares}
We have 
\begin{eqnarray*}
\mathrm{reg}(\square_5^{\psi^{-1}})
&=&-\frac{\mu_4 \mu_5-1}{\mu_4(\mu_5-1)}{\cD}_1
-\frac{1}{\mu_3\mu_4(\mu_5-1)}{\cD}_4,\\
\mathrm{reg}(\square_6^{\psi^{-1}})
&=&-\frac{\mu_2 \mu_5-1}{\mu_2(\mu_5-1)}{\cD}_1
-\frac{1}{\mu_1\mu_2(\mu_5-1)}{\cD}_3.
\end{eqnarray*}
\end{lemma}
\subsection{Circuit transformations}
\subsubsection{Generators of the fundamental group}
We give generators of the fundamental group $\pi_1(X,(\e,\e))$.
Let $\rho_1$ (and $\rho_3$, $\rho_5$) be a loop in 
$$L_x=\{(x,\e)\in \C^2\mid x\ne 0,1-\e,1\}$$
starting from $(x,y)=(\e,\e)$, 
approaching to the point $(x,y)=(0,\e)$,(and $(x,y)=(1-\e,\e)$, $(1,\e)$)
with $\mathrm{Im}(x)>0$, turning once around the point positively, 
and tracing back to $(\e,\e)$.
Let $\rho_2$ (and $\rho_4$) be a loop in 
$$L_y=\{(\e,y)\in \C^2\mid y\ne 0,1-\e,1\}$$
starting from $(x,y)=(\e,\e)$, approaching to the point $(x,y)=(\e,0)$,
(and $(x,y)=(\e,1)$) with $\mathrm{Im}(y)>0$, 
turning once around the point positively, 
and tracing back to $(\e,\e)$.
It is known that the loops $\rho_1,\dots,\rho_5$ generate $\pi_1(X,(\e,\e))$.
The circuit matrices of $\CM^\mu$ and $\cCM^\mu$ along $\rho_i$ will be denoted as
$$\CM^\mu_i=\CM^\mu_{\rho_i},\quad \cCM^{-\mu}_i=\cCM^{-\mu}_{\rho_i}
\quad (i=1,\dots,5).$$

\subsubsection{Expressions of the circuit transformations}
\begin{theorem}
\label{th:monod-rep}
{Under the assumption (3.1), the circuit transformations $\CM_i^\mu\in 
{\rm GL}(\Hl_2(T,\psi))$ $(i=1,\dots,5)$ are given as}
\begin{eqnarray}
\label{eq:monod1}
\CM_1^\mu(\De)&=&\mu_{12}^{-1}\De
-(\mu_{12}^{-1}-1)
\Big(\fI(\De,\cD_1),\fI(\De,\cD_4)\Big)
{H^\mu_{14}}^{-1} 
\begin{pmatrix}
\De_1\\
\De_4
\end{pmatrix},\\
\nonumber
\CM_2^\mu(\De)&=&\mu_{34}^{-1}\De
-(\mu_{34}^{-1}-1)
\Big(\fI(\De,\cD_1),\fI(\De,\cD_3)\Big)
{H^\mu_{13}}^{-1} 
\begin{pmatrix}
\De_1\\
\De_3
\end{pmatrix},
\end{eqnarray}
\begin{eqnarray*}
\CM_3^\mu(\De)&=&\De
-\frac{1-\mu_{245}}
{\fI(\De_2,\cD_2)}
\fI(\De,\cD_2)\De_2,\\
\CM_4^\mu(\De)&=&\De
-\frac{1-\mu_{145}}{\fI(\De_{145},\cD_{145})}
\fI(\De,\cD_{145})\De_{145}
,\\
\CM_5^\mu(\De)&=&\De
-\frac{1-\mu_{235}}{\fI(\De_{235},\cD_{235})}
\fI(\De,\cD_{235})\De_{235},
\end{eqnarray*}
where $\De$ is any element of  $\Hl_2(T,\psi)$,  
$H^\mu_{1j}$ $(j=3,4)$ is the submatrix of $H^\mu$ consisting of the 
$(1,1)$, $(1,j)$, $(j,1)$ and $(j,j)$ entries of $H^\mu$, and 
$$\De_{145}=\De_{2}+\square_5^\psi,\quad 
\De_{235}=\De_{2}+\square_6^\psi,$$
$$\cD_{145}=\cD_{2}+\mathrm{reg}(\square_5^{\psi^{-1}}),\quad 
\cD_{235}=\cD_{2}+\mathrm{reg}(\square_6^{\psi^{-1}}).$$
\end{theorem}

\begin{proof}
Under the condition $c,c'\notin \Z$,  the linear transformation $\CM_1^\mu$ 
satisfies the assumption of Lemma \ref{lem:duality} (2). In fact, 
a fundamental system of solutions can be given 
by the hypergeometric series $F_2$ multiplied by  
the power functions 
$$1,\quad  x^{1-c},\quad  y^{1-c'},\quad x^{1-c}y^{1-c'}.$$
Thus the eigenvalues of $\CM_1^\mu$ are $1$ and 
$e^{-2\pi\sqrt{-1}c}=\mu_{12}^{-1}$, and each of the eigenspaces 
is two dimensional. It is easy to see that the locally finite chains 
$\square_1$ and $\square_4$ are invariant under the deformation along $\rho_1$.
Hence the twisted cycles $\De_1$ and $\De_4$ span the 
eigenspace of $\CM_1^\mu$ of eigenvalue $1$. 
Similarly we can show that the twisted cycles $\cD_1$ and $\cD_4$ span the 
eigenspace of $\cCM_1^{-\mu}$ of eigenvalue $1$. 
By Lemma \ref{lem:duality} (2),  
the eigenspace of $\CM_1^\mu$ of eigenvalue $\mu_{12}^{-1}$ is 
$$\langle\cD_1,\cD_4\rangle^\perp=\{\De\in \Hl_2(T,\psi)\mid 
\fI(\De,\cD_1)=\fI(\De,\cD_4)=0\}.$$
It is easy to see that the right hand side $\;`\!{\CM_1^\mu}$ of 
(\ref{eq:monod1}) satisfies
$$\;`\!{\CM_1^\mu}(\Delta_j)=\Delta_j,\quad j=1,4,\qquad
\;`\!{\CM_1^\mu}(\Delta)=\,\mu_{12}^{-1}\Delta,\quad \Delta\in\langle\cD_1,\cD_4\rangle^\perp.$$
By Proposition \ref{prop:int-mat}, $\De_1,\dots,\De_4$ form a basis 
even in the case $c\in \Z$ or $c'\in \Z$. Thus the representation matrix 
$M_1^\mu$ of $\CM_1^\mu$ with respect to this basis is continuous on the 
parameters $a,b,b',c,c'$. 
On the other hand, the expression $\;\!{\CM_1^\mu}$ is also continuous 
since the factor $\mu_{12}-1$ in the denominator of 
$${H^\mu_{14}}^{-1}=\frac{(\mu_1-1)(\mu_2-1)}{\mu_{12}-1}\begin{pmatrix}
\dfrac{(\mu_4-1)(\mu_{45}-1)}{\mu_4(\mu_5-1)}
&
\dfrac{(\mu_4-1)(\mu_{345}-1)}
{\mu_4(\mu_5-1)}
\\[4mm]
\dfrac{\mu_4-1}{\mu_{34}(\mu_5-1)}
&
\dfrac{\mu_{34}-1}{\mu_{34}(\mu_5-1)}
\end{pmatrix}.
$$
cancels by $\mu_{12}^{-1}-1$.
Similarly we have the expression of $\CM_2^\mu$.

To study $\CM_5^\mu$, we work temporarily under the condition 
$\mu_{235}\not=1$. 
We decompose $\rho_5$ into 
$\bar\rho_5\cdot \rho_5^\circ\cdot \bar\rho_5^{-1}$, where 
$\bar\rho_5$ is the approach to $x=1$ and $\rho_5^\circ$ is the 
turning path.
We trace the deformation of the triangle $\square_{235}$ made by 
$\square_2$ and $\square_6$ along $\bar\rho_5$. After the deformation, 
this becomes a small triangle near the point $(t_1,t_2)=(1,0)$. 
We see the argument of $f_4=1-t_2$ on this triangle. 
Since $y=\e$ and $\im(x)>0$ in $\bar\rho_5$, $t_1>1$ in this triangle, and 
$$1-t_2=(\e-1)t_2+xt_1$$
on the line $\ell_5:f_5=0$, 
$f_4=1-t_2$ varies negative to positive via the upper-half space, 
i.e, $\arg(f_4)$ decreases by $\pi$ along $\bar\rho_5$.
Note that this change is compatible with our assignment 
of $\arg(f_4)$ on $\square_2$ and $\square_6$ in Table \ref{tab:arg}.   
Thus the twisted cycle $\Delta_{235}=\De_2+\square_6^{\psi}$ 
plays the role of a vanishing cycle as the line $\ell_5$ approaches 
the point $(1,0)$. 
Since $\rho_5^\circ$ corresponds to the move of $f_5$ turning around 
the point $(1,0),$
the cycle $\Delta_{235}$ is an eigenvector of 
$\CM_5^\mu$ of eigenvalue $\mu_{235}$.
We can similarly show that $\cD_{235}$ 
is an eigenvector of $\cCM_5^{-\mu}$ of eigenvalue $\mu_{235}^{-1}$. 
On the other hand, we can find three chambers not affected by the move of the line
$\ell_5$ along $\rho_5^\circ$. For example,  $\square_3$,  
$\{(t_1,t_2)\in \R_2\mid t_1<0,t_2<0\}$ and
$\{(t_1,t_2)\in \R_2\mid t_1>1,t_2>1\}$. 
Hence $\CM_5^\mu$ has three dimensional eigenspace of eigenvalue $1$.
Lemma \ref{lem:duality} (2) yields that this eigenspace is expressed as 
$$\langle\cD_{235} \rangle^\perp=\{\De\in \Hl_2(T,\psi)\mid 
\fI(\De,\cD_{235})=0\}.$$
So $\CM_5^\mu$ has desired eigenvalues and eigenspaces.
Since the factor 
$$
\frac{1-\mu_{235}}{\fI(\De_{235},\cD_{235})}
=(1-\mu_{235})\Big/\left(\frac{\mu_{235}-1}{(\mu_2-1)(\mu_3-1)(\mu_5-1)}\right)
=(\mu_2-1)(\mu_3-1)(\mu_5-1)
$$ 
is continuous on $\mu_{235}$ at $1$, 
the expression of $\CM_5^\mu$ is valid even in the case $\mu_{235}=1$. 
Similarly we have the expressions of $\CM_3^\mu$ and $\CM_4^\mu$.
\end{proof}

\begin{corollary}
\label{cor:monod-rep-dual}
Under the assumption (3.1), 
the circuit transformations $\cCM_i^{-\mu}\in 
{\rm GL}(H_2(T,\psi^{-1}))$ $(i=1,\dots,5)$ are given as
\begin{eqnarray*}
\cCM_1^{-\mu}(\cD)&=&\mu_{12}\cD
-(\mu_{12}-1)
\Big(\cD_1,\cD_4\Big)
\begin{pmatrix}
\fI(\De_1,\cD)
\\
\fI(\De_4,\cD)
\end{pmatrix}
{H^\mu_{14}}^{-1} ,\\
\cCM_2^{-\mu}(\cD)
&=&\mu_{34}\cD
-(\mu_{34}-1)
\Big(\cD_1,\cD_3\Big)
\begin{pmatrix}
\fI(\De_1,\cD)
\\
\fI(\De_3,\cD)
\end{pmatrix}
{H^\mu_{13}}^{-1},
\end{eqnarray*} 
\begin{eqnarray*}
\cCM_3^{-\mu}(\cD)&=&\cD
-\frac{1-\mu_{245}^{-1}}
{\fI(\De_2,\cD_2)}
\fI(\De_2,\cD)\cD_2,\\
\cCM_4^{-\mu}(\cD)&=&\cD
-\frac{1-\mu_{145}^{-1}}
{\fI(\De_{145},\cD_{145})}
\fI(\De_{145},\cD)\cD_{145},\\
\cCM_5^{-\mu}(\cD)&=&\cD
-\frac{1-\mu_{235}^{-1}}
{\fI(\De_{235},\cD_{235})}
\fI(\De_{235},\cD)\cD_{235},\\
\end{eqnarray*}
where $\cD$ is any element of $H_2(T,\psi^{-1})$.
\end{corollary}

\subsubsection{Circuit matrices}\label{circuit}
Let $M_i^\mu$ and $\wM_i^{-\mu}$ $(i=1,\dots,5)$ be the circuit matrices 
along the loop $\rho_i$ 
with respect to the basis $^t(\De_1,\dots,\De_4)$ 
of $\Hl_2(T,\psi)$, and to $(\cD_1,\dots,\cD_4)$ 
of $H_2(T,\psi^{-1})$, respectively. That is,  we have transformations 
$$
\begin{pmatrix}
\De_1\\
\vdots\\
\De_4
\end{pmatrix}
\mapsto 
M_i^\mu
\begin{pmatrix}
\De_1\\
\vdots\\
\De_4
\end{pmatrix}
,\quad 
(\cD_1,\dots,\cD_4)\mapsto 
(\cD_1,\dots,\cD_4)
\wM_i^{-\mu}
$$
by the continuation along $\rho_i$.

\begin{corollary}
\label{cor:monod-matrix}
The circuit matrices are expressed as
\begin{eqnarray*}
M_i^\mu&=&\lambda_iI_4-(\lambda_i-1)
H^\mu \widetilde R_i (R_iH^\mu \widetilde R_i)^{-1} R_i\quad (i=1,2),\\
M_j^\mu&=&I_4-(1-\lambda_j)
H^\mu \widetilde r_j (r_jH^\mu \widetilde r_j)^{-1} r_j\quad (j=3,4,5),
\end{eqnarray*}
\begin{eqnarray*}
\wM_i^{-\mu}&=&\frac{1}{\lambda_i}I_4-\big(\frac{1}{\lambda_i}-1\big)
\widetilde R_i (R_iH^\mu\;^t \widetilde R_i)^{-1} R_iH^\mu 
\quad (i=1,2),\\
\wM_j^{-\mu}&=&I_4-\big(1-\frac{1}{\lambda_j}\big)
\widetilde r_j (r_jH^\mu \widetilde r_j)^{-1} r_j H^\mu\quad (j=3,4,5),\\
\end{eqnarray*}
where 
$$\lambda_1=\mu_{12}^{-1},\quad 
\lambda_2=\mu_{34}^{-1},\quad 
\lambda_3=\mu_2\mu_4\mu_5,\quad 
\lambda_4=\mu_1\mu_4\mu_5,\quad 
\lambda_5=\mu_2\mu_3\mu_5,
$$
$$R_1=\begin{pmatrix}
1 & 0 & 0 & 0\\
0 & 0 & 0 & 1\\
\end{pmatrix}
,\quad
R_2=\begin{pmatrix}
1 & 0 & 0 & 0\\
0 & 0 & 1 & 0\\
\end{pmatrix}
,\quad r_3=(0,1,0,0),
$$
$$
r_4=\big(-\frac{\mu_{45}-1}{\mu_5-1},1,0,
-\frac{\mu_{345}-1}{\mu_5-1}\big),\quad 
r_5=\big(-\frac{\mu_{25}-1}{\mu_5-1},1,
-\frac{\mu_{125}-1}{\mu_5-1},0\big),
$$
$$
R_1=\;^t R_1,\quad R_2=\;^t R_2,\quad r_3=\;^t r_3,\quad 
\widetilde r_4=\begin{pmatrix}
\frac{-(\mu_{45}-1)}{\mu_4(\mu_5-1)}\\
 1\\
 0\\
\frac{-1}{\mu_{34}(\mu_5-1)}\\
\end{pmatrix},\quad
\widetilde r_5=\begin{pmatrix}
\frac{-(\mu_{25}-1)}{\mu_2(\mu_5-1)}\\
 1\\
\frac{-1}{\mu_{12}(\mu_5-1)}\\
 0\\
\end{pmatrix}.
$$
Their explicit forms are 
\begin{eqnarray*}
M_1^\mu&=&
\begin{pmatrix}
1 & 0 & 0 & 0\\
\frac{(\mu_1-1) (\mu_4 \mu_5-1)}{\mu_1 (\mu_5-1) } 
& \frac{1}{\mu_1 \mu_2} & 0 
& \frac{(\mu_1-1) (\mu_{345}-1)}{\mu_1(\mu_5-1)  }\\
\frac{-(\mu_2-1)}{\mu_1\mu_2} & 0 & \frac{1}{\mu_1 \mu_2} & 0\\
0 & 0 & 0 & 1
\end{pmatrix},\\
M_2^\mu&=&
\begin{pmatrix}1 & 0 & 0 & 0\\
\frac{(\mu_3-1) (\mu_2 \mu_5-1)}{\mu_3 (\mu_5-1)} 
& \frac{1}{\mu_3 \mu_4} 
& \frac{(\mu_3-1) (\mu_{125}-1)}{\mu_3 (\mu_5-1)} 
& 0\\
0 & 0 & 1 & 0\\
\frac{-(\mu_4-1)}{\mu_3\mu_4} & 0 & 0 & \frac{1}{\mu_3 \mu_4}\end{pmatrix},
\end{eqnarray*}
\begin{eqnarray*}
M_3^\mu&=&
\begin{pmatrix}1 &\mu_5-1 & 0 & 0\\
0 & \mu_{245} & 0 & 0\\
0 & 0 & 1 & 0\\
0 & 0 & 0 & 1\end{pmatrix},
\\
M_4^\mu&=&
\begin{pmatrix}
1-\mu_1 +\mu_{145} & -\mu_1(\mu_5-1) & 0 
& \mu_1 (\mu_{345} -1)\\
\frac{-(\mu_1-1) (\mu_4 \mu_5-1)}{\mu_5-1} & \mu_1 & 0 
& \frac{-(\mu_1-1) (\mu_{345} -1)}{\mu_5-1}\\
-(\mu_4 \mu_5-1) & \mu_5-1 & 1 
& -(\mu_{345} -1)\\
0 & 0 & 0 & 1\end{pmatrix},
\\ 
M_5^\mu&=&
\begin{pmatrix}1-\mu_3+\mu_{235} & - \mu_3(\mu_5-1) 
& \mu_3 (\mu_{125}-1) & 0\\
\frac{-(\mu_3-1) (\mu_2 \mu_5-1)}{\mu_5-1} & \mu_3 
& \frac{-(\mu_3-1) (\mu_{125}-1)}{\mu_5-1} & 0\\
0 & 0 & 1 & 0\\
-(\mu_2 \mu_5-1) & \mu_5-1  & 
-(\mu_{125}-1) & 1\end{pmatrix},
\end{eqnarray*}

\begin{eqnarray*}
\wM_1^{-\mu}&=&
\begin{pmatrix}
1 & \frac{-(\mu_1-1)(\mu_4\mu_5-1)}{\mu_4(\mu_5-1)}
&\mu_1(\mu_2-1)(\mu_{125}-1) &0\\
0 & \mu_1\mu_2 & 0& 0\\
0 & 0 & \mu_1\mu_2 & 0 \\
0 &\frac{-(\mu_1-1)}{\mu_3\mu_4(\mu_5-1)}& 0 &1
\end{pmatrix},
\\
\wM_2^{-\mu}&=&
\begin{pmatrix}
1 & \frac{-(\mu_3-1)(\mu_2\mu_5-1)}{\mu_2(\mu_5-1)} & 0
&\mu_3(\mu_4-1)(\mu_{345}-1)\\
0 & \mu_3\mu_4 & 0 &0 \\
0 & \frac{-(\mu_3-1)}{\mu_1\mu_2(\mu_5-1)}& 1 & 0\\
0 & 0 & 0 & \mu_3\mu_4
\end{pmatrix},
\end{eqnarray*}
\begin{eqnarray*}
\wM_3^{-\mu}&=&
\begin{pmatrix}1 &0 & 0 & 0\\
\frac{-(\mu_5-1)}{\mu_5} & \frac{1}{\mu_{245}} & 0 & 0\\
0 & 0 & 1 & 0\\
0 & 0 & 0 & 1\end{pmatrix},
\\
\wM_4^{-\mu}&=&
\begin{pmatrix}
\frac{1-\mu_4\mu_5+\mu_{145}}{\mu_{145}}& 
\frac{(\mu_1-1)(\mu_4\mu_5-1)}{\mu_1\mu_4(\mu_5-1)}&
\frac{(\mu_4\mu_5-1)(\mu_{125}-1)}{\mu_4\mu_5} & 0\\
\frac{\mu_5-1}{\mu_1\mu_5}& \frac{1}{\mu_1} &
\frac{-(\mu_5-1)(\mu_{125}-1)}{\mu_5}& 0\\
0 & 0 & 1 & 0\\
\frac{-1}{\mu_{134}\mu_5} &
\frac{\mu_1-1}{\mu_{134}(\mu_5-1)}&
\frac{\mu_{125}-1}{\mu_{345}} & 1
\end{pmatrix},
\\
\wM_5^{-\mu}&=&
\begin{pmatrix}
\frac{1-\mu_2\mu_5+\mu_{235}}{\mu_{235}}& 
\frac{(\mu_3-1)(\mu_2\mu_5-1)}{\mu_2\mu_3(\mu_5-1)}& 0& 
\frac{(\mu_2\mu_5-1)(\mu_{345}-1)}{\mu_2\mu_5} \\
\frac{\mu_5-1}{\mu_3\mu_5}& \frac{1}{\mu_3} &
0 &\frac{-(\mu_5-1)(\mu_{345}-1)}{\mu_5}\\
\frac{-1}{\mu_{123}\mu_5} &
\frac{\mu_3-1}{\mu_{123}(\mu_5-1)}&
1 &\frac{\mu_{345}-1}{\mu_{125}}\\ 
0 & 0 & 0 & 1\\
\end{pmatrix}.
\end{eqnarray*}
They satisfy 
$$M_i^{\mu} H^\mu  \wM_i^{-\mu}= H^\mu\quad (i=1,\dots,5).$$
\end{corollary}

\begin{proof}
We identify 
$$\De=\sum_{i=1}^4 z_i \De_i\in \Hl_2(T,\psi),\quad 
\cD
=\sum_{i=1}^4 \widetilde z_i \cD_i\in H_2(T,\psi^{-1})$$
with the row and column vectors 
$$z=(z_1,\dots,z_4),\quad \widetilde z=\begin{pmatrix}
\widetilde z_1\\
\vdots\\
\widetilde z_4
\end{pmatrix},$$
respectively. Note that 
$$\fI(\De,\cD)=zH^\mu\widetilde z.$$
Theorem \ref{th:monod-rep} yields these expressions.
We have
$$M_i^{\mu} H^\mu  \wM_i^{-\mu}= H^\mu\quad (i=1,\dots,5)$$
by Lemma \ref{lem:duality} (1).
\end{proof}

\begin{remark}
As a result, $M_i^\mu$ $(i=1,\dots,5)$ coincides with 
the circuit matrix with respect to the basis  $\De_1,\dots,\De_4$ 
by Theorem 7.1 in \cite{MY}.
\end{remark}
\section{Schwarz maps for $\mathcal{E}$ as the universal Abel-Jacobi maps}\label{Schmap}
In this section we introduce a system $\mathcal E$, and describe its Schwarz map, which is the main result of this paper.
\subsection{The system $\mathcal E$: a restriction of the system $E(3,6;1/3)$}\label{E36}
We introduce in this subsection a system $\mathcal E$, which is a system $E_2$ with specific parameters, and mention a reason why this system is of special interest. 

Let $X(3,6)$ be the configuration space of 
six lines $\ell_1,\dots,\ell_6$ in general position in the projective plane 
$\P^2=\{(p:q:r)\}$. We identify the space $X(3,6)$ with 
$$\left\{\left(\begin{array}{cccccc}1&0&1&x^1&x^2&0\\
0&1&1&x^3&x^4&0\\0&0&1&1&1&1\end{array}\right)\Bigg|
\textrm{ no }3\times3\textrm{ - minor vanish}\right\},$$
where $\ell_6$ is the line at infinity in the $pq$-plane 
given by $r=0$.
The system $$E(3,6;a),\quad a=(a_1,\dots,a_6),\quad a_1+\cdots+a_6=3$$
is generated by the linear differential equations which annihilate 
functions on $X(3,6)$ defined by the integral
$$\iint_{\rm a\ cycle}\prod_{j=1}^5f_j(x;p,q)^{a_j-1}dp\wedge dq. 
\qquad \left(
\begin{array}{l}
f_1=p,\quad f_2=q,\quad f_3=p+q+1,\\
f_4=x^1p+x^3q+1,\quad f_5=x^2p+x^4q+1.
\end{array}\right)
$$
The Schwarz map of the system $E(3,6;a)$ is studied 
(cf. \cite{MSY}, \cite{MSTY}) in two cases $a_j{=}1/2$ and $a_j\equiv1/6{\mod \Z}$. 
We have been interested in the case $a_j\equiv1/3 {\mod \Z}$.

On the other hand, let $X_2$ be the 2-dimensional stratum defined 
by $x^2=x^3=0$, which is the space of six lines such that the three lines 
$\{\ell_1,\ell_3,\ell_6\}$ meet at a point, 
the three lines $\{\ell_1,\ell_2,\ell_5\}$ meet at another point, 
and nothing further special occurs. 
It is known (\cite{MSY}) that the restriction of $E(3,6;a)$ onto $X_2$ is 
the Appell's hypergeometric system $E_3$, which is projectively equivalent 
(multiplying a function to the unknown) to
$$E_2(a_1,1-a_5,1-a_6,2-a_2-a_5,2-a_3-a_6;x,y),\quad x=1/x^2,\ y=1/x^4.$$
Setting $a=(4/3,1/3,1/3,1/3,1/3,1/3)$, we define
$$\mathcal{E}:=E_2\left(\frac43,{\bf \frac23},{\bf \frac43};x,y\right),
\quad {\rm where}\quad{\bf \frac23}=\left(\frac23,\frac23\right),
\ {\bf \frac43}=\left(\frac43,\frac43\right).$$   
We believe that the first step of understanding $E(3,6;a), a_j\equiv 1/3{\mod \Z}$ is the study of the system $\mathcal{E}$. 

\par\bigskip
The Schwarz map of a system is defined by  the ratio of linearly independent solutions. 
The main objective of this paper is the Schwarz map of the hypergeometric system $\mathcal{E}.$ The system $\mathcal{E}$ admits solutions stated in Proposition \ref{matome}. The next subsection gives a geometric background of understanding these solutions.

\subsection{A family of curves of genus 2}\label{genus2}
Consider a family of curves of genus 2 given as triple covers of $\P^1$:
$$C_t:S^3=s^2(1-s)(t-s)^2,\quad t:{\rm \ parameter},$$
 branching at four points $\{0,1,t,\infty\}$.
We choose two linearly independent holomorphic 1-forms:
$$\omega_1=s^{-2/3}(s-1)^{-1/3}(s-t)^{-2/3}ds,\quad\omega_2=s^{-1/3}(s-1)^{-2/3}(s-t)^{-1/3}ds,$$
and put
$$\varphi_1(s,t)=\int_0^s\omega_1(t),\quad \varphi_2(s,t)=\int_0^s\omega_2(t).$$ 
\par
For a fixed $t$, the {\it Abel-Jacobi map} for the curve $C_t$ is a multi-valued map
$$C_t\ni s \longmapsto \left(\varphi_1(s,t),\ \varphi_2(s,t)\right)\in\C^2;$$
It is a single-valued map to its Jacobian $\C^2/L$, where $L$ is a lattice generated by its periods: integrals over possible loops with base $s=0$:
$$\left(\varphi_1(0,t),\ \varphi_2(0,t)\right).$$
\subsection{The Schwarz map of $\mathcal{E}$}
Proposition \ref{matome} for $a=c=c'=4/3,b=b'=2/3$ 
implies that after the coordinate change 
$$\xx=\frac{-x}{1-x},\quad  \zz=\xx\yy,\quad{\rm where}\quad \yy=\frac{-y}{1-y},$$
and the change of unknown: $u\to (1-y)^{2/3}(1-x)^{2/3}u$, 
two linearly independent solutions of $$E\left({\bf \frac23},\frac43;\zz\right)$$
and the two indefinite integrals
$$\int_0^\xx s^{-2/3}(1-s)^{-1/3}(\zz-s)^{-2/3}ds,\quad\zz^{-1/3}\int_0^\xx s^{-1/3}(1-s)^{-2/3}(\zz-s)^{-1/3}ds$$
form a set of fundamental solutions of $\mathcal{E}$. 

On the other hand, the integral representation of the Gauss hypergeometric equation given in Section 1 asserts that the integral above along any {\it closed} path gives a solution of $E({\bf{2/3}},4/3)$. Thus we find that  the Schwarz map of $\mathcal{E}$ is the totality of the Abel-Jacobi map of the family $\{C_t\}$ after a slight modification (multiplying $t^{-1/3}$ to the second coordinate).

Thus we get
\begin{theorem}
\label{th:gen-Schwarz} 
If we change the coordinates $(x,y)$ of $\mathcal{E}=E_2(4/3,{\bf2/3},{\bf4/3};x,y)$ as
$$s(=\xx)=\frac{-x}{1-x},\quad  t(=\zz)=\xx\yy,\quad{\rm where}\quad \yy=\frac{-y}{1-y},$$
the Schwarz map of the system $\mathcal{E}$ is equivalent to the 
projectivization of the family of the  Abel-Jacobi map of 
the family $\{C_t\}$ of curves of genus 2, explicitly given as
$$\mathcal{S}_2:\dsp\bigcup_{t\in \C-\{0,1\}}C_t\ni(s,t)\longmapsto \varphi_1(0,t):t^{-1/3}\varphi_2(0,t):\varphi_1(s,t):t^{-1/3}\varphi_2(s,t)\in\P^3.$$
The latter two  $\varphi_1(s,t)$ and $t^{-1/3}\varphi_2(s,t)$ are $f_1$ and $f_2$ in Section 2.
The map by means of the former two
$$\mathcal{S}_1:\P^1-\{0,1,\infty\}\ni t\longmapsto \varphi_1(0,t):\ t^{-1/3}\varphi_2(0,t)\in \P^1$$
is the Schwarz map of the hypergeometric equation $E({\bf 2/3},4/3)$. Its image is a disc, and the inverse map of $\mathcal{S}_1$ is single-valued automorphic function on the disc with respect to the triangle group of type $[3,\infty,\infty]$; in other words, the disc is tessellated by Schwarz triangles of type $[3,\infty,\infty]$. 
\par\noindent
The image surface under $\mathcal{S}_2$ can be regarded as lying in a fiber bundle with the $\mathcal{S}_1$-image disc as its base and the Jacobian variety of $C_t$ as the fiber on the image point $\mathcal{S}_1(t)$. 
\end{theorem}
{A triangle of type $[p,q,r]$ is a hyperbolic triangle with angles $\pi/p,\pi/q$ and $\pi/r$; the above triangle has angles $\pi/3,0$ and $0$. The triangle group of type $[p,q,r]$ is the group consisting of the even products of the reflections with the sides of the triangle of type $[p,q,r]$ as axes. It is known that the triangle group of type $[3,\infty,\infty]$ is conjugate to the congruence subgroup
$$\Gamma_1(3)=\left\{\left(\begin{array}{cc}a&b\\c&d\end{array}\right)\in 
{\rm SL}_2(\Z)\mid a-1,d-1,c\equiv 0\mod 3\right\}.$$ 
For arithmetic triangle groups, see \cite{T}.}

{Other than this family of curves, there are two families of curves of genus 2 branching at four points in $\P^1$; see Appendix 2}.
\subsection{Monodromy group of $\mathcal E$}\label{special}
From monodromy side, Theorem \ref{th:gen-Schwarz} can be understood as follows. Define $M_i$ and $\wM_i$ $(i=1,\dots,5)$ 
by substituting $\mu_1=\cdots=\mu_5=\omega^2=\frac{-1-\sqrt{-3}}{2}$
into $M_i^\mu$ and $\wM_i^{-\mu}$ defined in \S \ref{circuit}, respectively.  
They are the circuit matrices for $\mathcal{E}$ 
with respect to 
$$^t\Big(\iint_{\square_1}\psi_\CE dt_1dt_2,\dots,\iint_{\square_4}\psi_\CE dt_1dt_2\Big)$$
and those for $\mathcal{E}^\vee=E_2(-\frac{4}{3},
-\frac{2}{3},-\frac{2}{3},-\frac{4}{3},-\frac{4}{3})$ 
with respect to 
$$\Big(\iint_{\square_1}\psi_\CE^{-1} dt_1dt_2,\dots,\iint_{\square_4}\psi_\CE^{-1} dt_1dt_2
\Big),$$
where
$$\psi_\CE
=\frac{1}{\sqrt[3]{t_1(1-t_1)t_2(1-t_2)(1-t_1x-t_2y)}}
.$$

\begin{corollary}
\label{cor:monodromy}
We have 
$$
\begin{array}{ll}
M_1=\begin{pmatrix}
1 & 0 & 0 & 0\\
-2 \omega-1 & -\omega-1 & 0 & 0\\
-2 \omega-1 & 0 & -\omega-1 & 0\\
0 & 0 & 0 & 1
\end{pmatrix},\quad
& 
\wM_1=\begin{pmatrix}
1 & 2\omega+1 & 0 & 0\\
0 & \omega & 0 & 0\\
0 & 0 & \omega & 0\\
0 & \omega+1 & 0 & 1
\end{pmatrix},
\\
M_2=\begin{pmatrix}
1 & 0 & 0 & 0\\
-2 \omega-1 & -\omega-1 & 0 & 0\\
0 & 0 & 1 & 0\\
-2 \omega-1 & 0 & 0 & -\omega-1
\end{pmatrix},\quad
&
\wM_2=\begin{pmatrix}
1 & 2\omega+1 & 0 & 0\\
0 & \omega & 0 & 0\\
0 & \omega+1 & 1 & 0\\
0 & 0 & 0 & \omega
\end{pmatrix},
\\
\end{array}
$$
$$
\begin{array}{ll}
M_3=\begin{pmatrix}
1 & -\omega-2 & 0 & 0\\
0 & 1 & 0 & 0\\
0 & 0 & 1 & 0\\
0 & 0 & 0 & 1
\end{pmatrix},\quad
&
\wM_3=\begin{pmatrix}
1 & 0 & 0 & 0\\
\omega-1 & 1 & 0 & 0\\
0 & 0 & 1 & 0\\
0 & 0 & 0 & 1
\end{pmatrix},

\\
M_4=\begin{pmatrix}
\omega+3 & -2 \omega-1 & 0 & 0\\
-\omega+1 & -\omega-1 & 0 & 0\\
-\omega+1 & -\omega-2 & 1 & 0\\
0 & 0 & 0 & 1
\end{pmatrix},\quad
& 
\wM_4=\begin{pmatrix}
-\omega+2 & \omega+2 & 0 & 0\\
2\omega+1 & \omega & 0 & 0\\
0 & 0 & 1 & 0\\
-\omega & 1 & 0 & 1
\end{pmatrix},
\\
M_5=\begin{pmatrix}
\omega+3 & -2 \omega-1 & 0 & 0\\
-\omega+1 & -\omega-1 & 0 & 0\\
0 & 0 & 1 & 0\\
-\omega+1 & -\omega-2 & 0 & 1
\end{pmatrix},\quad
&
\wM_5=\begin{pmatrix}
-\omega+2 & \omega+2 & 0 & 0\\
2\omega+1 & \omega & 0 & 0\\
-\omega & 1 & 1 & 0\\
0 & 0 & 0 & 1
\end{pmatrix}.

\\
\end{array}
$$
They satisfy 
$$M_i H \wM_i=H, \quad (i=1,\dots,5),\qquad H
=\frac{-1}{3}
\begin{pmatrix}
1& \omega& 0 &0\\
-\omega-1& 0& 0 &0\\
-\omega-1& 0& \sqrt{-3} &0\\ 
-\omega-1& 0& 0 &\sqrt{-3}\\ 
\end{pmatrix}.
$$
\end{corollary}

By Proposition \ref{prop:inv-subspace} and Remark \ref{rem:kernel},
the subspace spanned by solutions
$\iint_{\square_1} \psi_\CE dt_1dt_2$ and
$\iint_{\square_2} \psi_\CE dt_1dt_2$ is invariant under the monodromy
representation. In fact,
the top-left $2\times2$ block matrices $M_i'$ of $M_i$ $(i=1,\dots,5)$
act on this space.
Note that $M'_1=M'_2$, $M'_4=M'_5$.
Let $G$ be the group generated by $M'_1$,  $M'_3$ and   $M'_5$.
The group $G$ is isomorphic to the triangle group $[3,\infty,\infty]$,
and is contained in
the unitary group
$$\Big\{g\in \mathrm{GL}_2(\Z[\omega])\mid g H' \;^t \overline{ g}=H'=
\begin{pmatrix}
-1 & -\omega \\
\omega+1 &0
\end{pmatrix}
\Big\}.
$$
By a matrix
$$P=\begin{pmatrix}
1 &1 \\
0& -2-\omega\\
\end{pmatrix},
$$
the Hermite matrix $H'$ and circuit matrices $M'_i$ $(i=1,3,5)$
are transformed as
$$
PH'\;^t \overline{P}=
\sqrt{-3}\begin{pmatrix}
0 & 1 \\ -1 & 0
\end{pmatrix},
$$
$$PM'_1P^{-1}=\omega\begin{pmatrix}
-2 & -1 \\3&1
\end{pmatrix},\quad
PM'_3P^{-1}=\begin{pmatrix}
1 & 1 \\0&1
\end{pmatrix},\quad
PM'_5P^{-1}=\begin{pmatrix}
4 & 3 \\-3&-2
\end{pmatrix}.
$$
Hence the projectivization of $G$ is isomorphic to
the congruence subgroup $\Gamma_1(3)$ of $\mathrm{SL}_2(\Z)$
and the ratio
$\iint_{\square_1} \psi_\CE dt_1dt_2/\iint_{\square_2} \psi_\CE dt_1dt_2$
can be regarded as the map $\mathcal{S}_1$ in Theorem \ref{th:gen-Schwarz}
and as an element of the upper-half space.

\begin{appendix}
\section{Restriction of the system $E(3,6;a)$ on a 3-dimensional stratum}
\label{3-dim-S}
Let $X_3$ denote the stratum defined by $x^2=0$ 
and let us restrict the system $E(3,6;a)$ to this stratum, 
which is the space of six lines such that the three lines 
$\{\ell_1,\ell_3,\ell_6\}$ meet at a point, 
and nothing further special occurs. 
This system is denoted by $E(3,6;a)|X_3$ or $EX_3$. 
Little is known about this system. 

{Before stating the proposition in this section, we briefly recall the 
Appell-Lauricella's system $E_D^{(3)}(a,b_1,b_2,b_3,c;y^1,y^2,y^3):$
\begin{eqnarray*}
 && \delta_i(\delta + c-1)u - y^i(\delta_i+b_i)(\delta+a)u=0, \\
&& y^i(\delta_i+b_i)\delta_ju - y^j(\delta_j+b_j)\delta_iu=0,
\end{eqnarray*}
where $(y^1,y^2, y^3)$ are variables, and 
$\delta_i=y^i \partial/\partial y^i$ and $\delta=\delta_1+\delta_2+\delta_3$.
This is a 3-variable version of the Appell's $E_1$. It admits  solutions given by a power series
$$F_D( a,b_1,b_2,b_3,c;y^1,y^2,y^3)=\sum_{n_1,n_2,n_3}^{\infty}\frac{(a,n_{123})(b_1,n_1)(b_2,n_2)(b_3,n_3)}{(c,n_{123})n_1!n_2!n_3!}(y^1)^{n_1}(y^2)^{n_2}(y^3)^{n_3},$$
where $n_{123}=n_1+n_2+n_3$, and by an integral
$$\int_0^1t^{a-1}(1-t)^{c-a-1}(1-ty^1)^{-b_1}(1-ty^2)^{-b_2}(1-ty^3)^{-b_3}dt.$$The collection of solutions is denoted by $S_D^{(3)}(a,b_1,b_2,b_3,c;y^1,y^2,y^3)$}

In this section, we prove the following proposition.

\begin{proposition}\label{X3FD}
If $a_2+a_4+a_5=1$, then the system $E(3,6;a)|X_3$ is
reducible and has a subsystem 
isomorphic to the Appell-Lauricella's system $E_D^{(3)}=E_D^{(3)}$ in 3 variables 
with 4 free parameters. More precisely, the collection of the solutions of $E(3,6;a)|X_3$ includes
\[
(1-x^1)^{-a_2}S_D^{(3)}\left(a_3,a_4,1-a_6,a_2,1+a_3-a_5; x^3, x^4, 
\dsp{x^3-x^1\over 1-x^1}\right).
\] 
Note $a_1+\cdots+a_6=3$.
\end{proposition}

If we apply the proposition under the further restriction $x^3=0$, 
we find a subsystem isomorphic to $E_1$ in $E_2$, 
which is equivalent to Proposition 1.1. 

{We give three ``proof''s: one using power series,
one using integral representations, 
and one manipulating differential equations.}

\subsection{Power series}
{It is known that the system $E(3,6;a)|_{X_3}$
has a solution given by the series
\[
F_{X_3}(a_2,a_3,a_4,a_5,a_6;x):=\sum_{n_1,n_3,n_4=0}^\infty
{(a'_5,n_{13})(a'_6,n_4)(a_2,n_1)(a_3,n_{34})
\over (a_2+a_3+a_4,n_{134})n_1!n_3!n_4!}\
(x^1)^{n_1}(x^3)^{n_3}(x^4)^{n_4},
\]
where $x=(x^1,x^3,x^4),n_{13}=n_1+n_3, n_{34}=n_3+n_4, n_{134}=n_1+n_3+n_4, a'_5=1-a_5, a'_6=1-a_6$; refer to \cite[p.47]{MSY}. A computation shows that the identity
$$F_{X_3}(a_2,a_3,a_4,a_5,a_6;x)=(1-x^1)^{-a_2}F_D^{(3)}\left(a_3,a_4,1-a_6,a_2,1+a_3-a_5; x^3, x^4, 
\dsp{x^3-x^1\over 1-x^1}\right)$$
holds if and only if $a_2+a_4+a_5=1.$}
\subsection{Integral representation}
We manipulate the integral
$$\iint p^{a_1-1}q^{a_2-1}r^{a_3-1}
(p+q+r)^{a_4-1}(p+x_1q+x_3r)^{a_5-1}
(p+x_4r)^{a_6-1}dp\wedge dq.$$
The three lines
$$\ell_1: p=0,\quad \ell_3: r=0,\quad \ell_6:p+x_4r=0$$
meet at $0:1:0$.
Introduce new coordinate $Q$ 
$$q=\frac{(p+x_3)(p+1)Q}D,\quad D:=x_1(p+1)-p-x_3-x_1(p+1)Q,$$
which send $q=0,-p-1, -\frac{p+x_3}{x_1}$ to $Q=0,1,\infty$.
Since 
$$q+p+1=\frac{(p+1)N(Q-1)}D,\quad x_1q+p+x_3=-\frac{(p+x_3)N}D,$$
$$dq=-\frac{(p+x_3)(p+1)NdQ}{D^2}+*dp,\quad N=(1-x_1)p+x_3-x_1,$$
we have
$$f=-\iint p^{a_1-1}\{(p+x_3)(p+1)Q)\}^{a_2-1}
\{(p+1)N(Q-1)\}^{a_4-1}$$
$$\times\{-(p+x_3)N\}^{a_5-1}(p+x_4)^{a_6-1}(p+x_3)(p+1)N\cdot D^{1-a_{245}}dp\wedge dQ,$$
where $a_{245}=a_2+a_4+a_5$.
If  
$$a_{245}=1\qquad ({\rm Note\ }:\{2,4,5\}=\{1,\dots,6\}-\{1,3,6\}),$$
then the double integral above becomes the product of the Beta integral
$$\int Q^{a_2-1}(Q-1)^{a_4-1}dQ$$
and the integral
$$\int p^{a_1-1}(p+1)^{a_2+a_4-1}(p+x_3)^{a_2+a_5-1}(p+x_4)^{a_6-1}\{(1-x_1)p+x_3-x_1\}^{a_4+a_5-1}dp,$$
which can be written as 
$$(1-x_1)^{-a_2}\int p^{a_1-1}(p+1)^{-a_5}(p+x_3)^{-a_4}(p+x_4)^{a_6-1}\left(p+\frac{x_3-x_1}{1-x_1}\right)^{-a_2}dp.$$
On the other hand, {the integral 
$$\int p^{b_1+b_2+b_3-c}(p-1)^{c-a-1}(p-y^1)^{-b_1}(p-y^2)^{-b_2}(p-uy^3)^{-b_3}dp$$
solves $E^{(3)}_D(a,b_1,b_2,b_3,c;y^1,y^2,y^3)$.} By solving the system
$$a_1-1=b_1+b_2+b_3-c,\quad -a_5=c-a-1,\quad -a_4=b_1, \quad a_6-1=b_2,\quad -a_2=-b_3,$$
we complete the proof of Proposition \ref{X3FD}
\subsection{System of differential equations}\label{secondproof} 
We manipulate the system $EX_3$ (given in [MSY, p.24]):
\begin{eqnarray*}
&& (\theta + a_{234} -1)\theta_1 w - x^1(\theta_1+\theta_3 + 1-a_5)
(\theta_1+a_2)w=0, \\
&& (\theta + a_{234}-1)\theta_3w 
   - x^3(\theta_1+\theta_3+1-a_5)(\theta_3+\theta_4+a_3)w=0, \\
&& (\theta + a_{234}-1)\theta_4w 
   - x^4(\theta_4+1-a_6)(\theta_3+\theta_4+a_3)w=0, \\
&& x^3(\theta_1+\theta_3+1-a_5)\theta_4 w 
   - x^4(\theta_4 + 1 -a_6)\theta_3 w =0, \\
&& x^1(\theta_1+a_2)\theta_3 w - x^3(\theta_3+\theta_4+a_3)\theta_1w=0,
\end{eqnarray*}
where $\theta_1=x^1\partial/\partial x^1$, 
$\theta_3=x^3\partial/\partial x^3$, $\theta_4=x^4\partial/\partial x^4$, 
and $\theta=\theta_1+\theta_3+\theta_4$, and the Appell-Lauricella system
$E_D^{(3)}=E_D^{(3)}(a,b_1,b_2,b_3,c;y^1,y^2,y^3)$.
We show that the system $E_D^{(3)}$ is a subsystem of $EX_3$ when $a_2+a_4+a_5=1$.

Change the unknown $w$ of $EX_3$ into $u$ by
\[ w=(1-x^1)^{-p}u,\]
and the variables $(x^1,x^3,x^4)$ into $(x,y,z)$ as
\[x=x^3, \qquad y=x^4, \qquad z={x^3-x^1\over 1-x^1}.\]
Then $EX_3$ can be written as $L_i u=0,\ 1\le i \le 5,$
where
\begin{eqnarray*}
&& L_1:=\left(\delta_x+\delta_y+{1-z\over 1-x}\delta_z+h+ a_{234} -1\right)
        (f\delta_z +h) \\
&&\qquad\qquad   - {x-z\over 1-z}\left(\delta_x+{1-z\over 1-x}\delta_z + h + 1-a_5\right) (f\delta_z + h+ a_2), \\
&& L_2:=\left(\delta_x+\delta_y+{1-z\over 1-x}\delta_z+h+ a_{234} -1\right)
   (\delta_x+g\delta_z) \\
&& \qquad \qquad   - x\left(\delta_x+{1-z\over 1-x}\delta_z + h+ 1-a_5\right)
   (\delta_x+\delta_y+g\delta_z+a_3), \\
&& L_3:=\left(\delta_x+\delta_y+{1-z\over 1-x}\delta_z+h+ a_{234} -1\right)\delta_y \\
&&\qquad\qquad   - y(\delta_y +1-a_6)(\delta_x+\delta_y+g\delta_z+a_3), \\
&& L_4:=x\left(\delta_x+{1-z\over 1-x}\delta_z + h+ 1-a_5\right)\delta_y
  - y(\delta_y + 1 -a_6)(\delta_x + g\delta_z), \\
&& L_5:={x-z\over 1-z}(f\delta_z + h + a_2)(\delta_x+g\delta_z)
  - x(\delta_x+\delta_y+g\delta_z+a_3)(f\delta_z +h),
\end{eqnarray*}
where
\[ f={(z-x)(1-z)\over z(1-x)}, \qquad g={x(1-z)\over z(1-x)}.\]
Write the system $E_D^{(3)}$ as $E_{xy}u=0,\dots, E_{zz}u=0$, where
\begin{eqnarray*}
& E_{xy}:=\delta_{xy} -( b_2y\delta_x - b_1x\delta_y)/(x-y), & \\
& E_{xz}:=\delta_{xz} -(b_3z\delta_x - b_1x\delta_z)/(x-z), & \\
& E_{yz}:=\delta_{yz} -( b_3z\delta_y - b_2y\delta_z)/(y-z), & \\
& E_{xx}:=\delta_{xx}+\delta_{xy}+\delta_{xz} -(
((a+b_1)x+1-c)\delta_x + b_1x(\delta_y+\delta_z+a))/(1-x), & \\
& E_{yy}:=\delta_{yy}+\delta_{xy}+\delta_{yz} -(
((a+b_2)y+1-c)\delta_y + b_2y(\delta_x+\delta_z+a))/(1-y), & \\
& E_{zz}:=\delta_{zz}+\delta_{xz}+\delta_{yz} -(
((a+b_3)z+1-c)\delta_z + b_3z(\delta_x+\delta_y+a))/(1-z). &
\end{eqnarray*}
Eliminating the second derivatives in $L_j$ by using $E_{**}$, we see that $L_j$ are linear combination, over $\C(x,y,z)$, of the $E_{**}$'s if and only if 
\[ a_2+a_4+a_5=1, (\leftrightarrow c=a+b_1+b_3)\]
and  
\[ p=a_2,\quad
a=a3, \quad b_1= a_4, \quad b_2= 1-a_6,\quad b_3=a_2, \quad c=1+a_3-a_5.\]
This completes the proof of Proposition \ref{X3FD}.
\par\medskip\noindent
{\bf Remark:} Actually we have, under the condition $c=a+b_1+b_3$,
$$\langle L_1,L_2,L_5\rangle =\langle E_{xx},E_{xz},E_{zz}\rangle,$$$$
\langle L_3,L_4\rangle =\langle E_{yy}-\frac{(z-x)(z-y)}{z(1-x)(1-y)}E_{yz},\ 
E_{yy}-\frac{(x-z)(x-y)}{x(1-z)(1-y)}E_{yx}
\rangle. $$
\section{Families of curves of genus 2}\label{genus2}
We encountered  a family of curves $C_t$ of genus 2 given as triple covers of $\P^1$. This is the Case 3 in the following Proposition.
\begin{proposition} A cyclic cover of $\P^1$ branching at four points is 
of genus 2 only in three cases:
$$\begin{array}{llll}
{\rm Case\ 3:}&3\ &{\rm fold\ cover\ with\ indices}\quad &3,\ 3,\ 3,\ 3,\\
{\rm Case\ 6:}&6\ &{\rm fold\ cover\ with\ indices}\quad &2,\ 2,\ 3,\ 3,\\
{\rm Case\ 4:}&4\ &{\rm fold\ cover\ with\ indices}\quad &2,\ 2,\ 4,\ 4. 
\end{array}$$
\end{proposition}
Indeed, since the $n$ fold cyclic cover $C$ of $\P^1$ branching 
at four points with indices $k_1,\dots,k_4$ has Euler characteristic
$$2n-\sum_{i=1}^4\frac{n}{k_i}(k_i-1),$$
if we assume the genus of $C$ is two (Euler characteristic of $C$ is $-2$), 
we have
$$\sum_{i=1}^4\frac1{k_i}+\frac2n=2,\quad {\rm l.c.m.}(k_1,\dots,k_4)=n;$$
it is easy to see that only three cases above are possible. 
\par\medskip
The three cases can be realized by the following families of curves:
$$\begin{array}{llll}
{\rm Case\ 3:}&C_t^{(3)}: &S^3=s^2(1-s)(t-s)^2,&\quad t:{\rm \ parameter},\\[2mm]
{\rm Case\ 6:}&C_t^{(6)}: &S^6=s^2(1-s)^4(t-s)^3,&\quad t:{\rm \ parameter},\\[2mm]
{\rm Case\ 4:}&C_t^{(4)}: &S^4=s^2(1-s)^2(s-t),&\quad t:{\rm \ parameter}.
\end{array}$$
Note that the double cover of the base space of Case 6 branching 
at the two points of index 2 is equivalent to Case 3.
\end{appendix}

\end{document}